\documentclass{amsart}
\usepackage[latin1]{inputenc}
\usepackage{amsmath,amssymb,latexsym}
\usepackage{amsrefs}
\usepackage{subfigure}
\usepackage[usesnames,svgnames]{xcolor}
\usepackage[sc]{mathpazo}
\usepackage{tikz,pgf}
  \usetikzlibrary{arrows}
\usepackage[colorlinks=true,
	    urlcolor=MidnightBlue,
	    citecolor=DarkGreen]{hyperref}
\usepackage{multirow,bigdelim}
\usepackage{longtable}
\newcommand{\nodea}[2]{
  \draw(#1,#2) node[circle,fill=black,scale=.6]{};}
\newcommand{\nodeb}[2]{
  \draw(#1,#2) node[circle,draw=black,scale=.6]{};}
\newcommand{\nodeba}[2]{
  \draw(#1,#2) node[circle,draw=black,scale=.6]{};
  \draw(#1,#2) node{\tiny $1$};}
\newcommand{\nodebb}[2]{
  \draw(#1,#2) node[circle,draw=black,scale=.6]{};
  \draw(#1,#2) node{\tiny $2$};}
\newcommand{\nodebc}[2]{
  \draw(#1,#2) node[circle,draw=black,scale=.6]{};
  \draw(#1,#2) node{\tiny $3$};}

\newcommand{\sOne}[2]{
	\begin{tikzpicture}[rotate=#1,scale=#2]
		\nodea{1}{1.5};
		\nodeb{1}{.5};
		\nodeba{.5}{1};
		\nodebb{1}{1};
		\nodebc{1.5}{1};
		\draw(.9,.6) -- (.6,.9);
		\draw(1,.6) -- (1,.9);
		\draw(1.1,.6) -- (1.4,.9);
		\draw(.6,1.1) -- (.9,1.4);
		\draw(1,1.1) -- (1,1.4);
		\draw(1.4,1.1) -- (1.1,1.4);
	\end{tikzpicture}
}
\newcommand{\sTwo}[2]{
	\begin{tikzpicture}[rotate=#1,scale=#2]
		\nodea{1}{1.5};
		\nodeb{1}{.5};
		\nodeba{.5}{1};
		\nodebb{1}{1};
		\nodebc{1.5}{1};
		\draw(.9,.6) -- (.6,.9);
		\draw(1,.6) -- (1,.9);
		\draw(1.1,.6) -- (1.4,.9);
		\draw(1,1.1) -- (1,1.4);
		\draw(1.4,1.1) -- (1.1,1.4);
	\end{tikzpicture}
}
\newcommand{\sThree}[2]{
	\begin{tikzpicture}[rotate=#1,scale=#2]
		\nodea{1}{1.5};
		\nodeb{1}{.5};
		\nodeba{.5}{1};
		\nodebb{1}{1};
		\nodebc{1.5}{1};
		\draw(.9,.6) -- (.6,.9);
		\draw(1,.6) -- (1,.9);
		\draw(1.1,.6) -- (1.4,.9);
		\draw(.6,1.1) -- (.9,1.4);
		\draw(1.4,1.1) -- (1.1,1.4);
	\end{tikzpicture}
}
\newcommand{\sFour}[2]{
	\begin{tikzpicture}[rotate=#1,scale=#2]
		\nodea{1}{1.5};
		\nodeb{1}{.5};
		\nodeba{.5}{1};
		\nodebb{1}{1};
		\nodebc{1.5}{1};
		\draw(.9,.6) -- (.6,.9);
		\draw(1,.6) -- (1,.9);
		\draw(1.1,.6) -- (1.4,.9);
		\draw(.6,1.1) -- (.9,1.4);
		\draw(1,1.1) -- (1,1.4);
	\end{tikzpicture}
}
\newcommand{\sFive}[2]{
	\begin{tikzpicture}[rotate=#1,scale=#2]
		\nodea{1}{1.5};
		\nodeb{1}{.5};
		\nodeba{.5}{1};
		\nodebb{1}{1};
		\nodebc{1.5}{1};
		\draw(.9,.6) -- (.6,.9);
		\draw(1,.6) -- (1,.9);
		\draw(1.1,.6) -- (1.4,.9);
		\draw(1.4,1.1) -- (1.1,1.4);
	\end{tikzpicture}
}
\newcommand{\sSix}[2]{
	\begin{tikzpicture}[rotate=#1,scale=#2]
		\nodea{1}{1.5};
		\nodeb{1}{.5};
		\nodeba{.5}{1};
		\nodebb{1}{1};
		\nodebc{1.5}{1};
		\draw(.9,.6) -- (.6,.9);
		\draw(1,.6) -- (1,.9);
		\draw(1.1,.6) -- (1.4,.9);
		\draw(1,1.1) -- (1,1.4);
	\end{tikzpicture}
}
\newcommand{\sSeven}[2]{
	\begin{tikzpicture}[rotate=#1,scale=#2]
		\nodea{1}{1.5};
		\nodeb{1}{.5};
		\nodeba{.5}{1};
		\nodebb{1}{1};
		\nodebc{1.5}{1};
		\draw(.9,.6) -- (.6,.9);
		\draw(1,.6) -- (1,.9);
		\draw(1.1,.6) -- (1.4,.9);
		\draw(.6,1.1) -- (.9,1.4);
	\end{tikzpicture}
}
\newcommand{\sEight}[2]{
	\begin{tikzpicture}[rotate=#1,scale=#2]
		\nodea{1}{1.5};
		\nodeb{1}{.5};
		\nodeba{.5}{1};
		\nodebb{1}{1};
		\nodebc{1.5}{1};
		\draw(.9,.6) -- (.6,.9);
		\draw(1,.6) -- (1,.9);
		\draw(1.1,.6) -- (1.4,.9);
		\draw(.9,.6) .. controls (.75,.8) and (.75,1.2) .. (.9,1.4);
	\end{tikzpicture}
}
\newcommand{\sNine}[2]{
	\begin{tikzpicture}[rotate=#1,scale=#2]
		\nodea{1}{1.5};
		\nodeb{1}{.5};
		\nodeba{.5}{1};
		\nodebb{1}{1};
		\nodebc{1.5}{1};
		\draw(.9,.6) -- (.6,.9);
		\draw(1,.6) -- (1,.9);
		\draw(1.1,.6) -- (1.4,.9);
	\end{tikzpicture}
}
\newcommand{\sTen}[2]{
	\begin{tikzpicture}[rotate=#1,scale=#2]
		\nodea{.75}{1};
		\nodeb{1}{.5};
		\nodeba{.5}{1.5};
		\nodebb{1}{1.5};
		\nodebc{1.5}{1.5};
		\draw(.95,.6) -- (.55,1.4);
		\draw(1,.6) -- (1,1.4);
		\draw(1.05,.6) -- (1.45,1.4);
	\end{tikzpicture}
}
\newcommand{\sEleven}[2]{
	\begin{tikzpicture}[rotate=#1,scale=#2]
		\nodea{1}{1};
		\nodeb{1}{.5};
		\nodeba{.5}{1.5};
		\nodebb{1}{1.5};
		\nodebc{1.5}{1.5};
		\draw(.95,.6) -- (.55,1.4);
		\draw(1,.6) -- (1,1.4);
		\draw(1.05,.6) -- (1.45,1.4);
	\end{tikzpicture}
}
\newcommand{\sTwelve}[2]{
	\begin{tikzpicture}[rotate=#1,scale=#2]
		\nodea{1.25}{1};
		\nodeb{1}{.5};
		\nodeba{.5}{1.5};
		\nodebb{1}{1.5};
		\nodebc{1.5}{1.5};
		\draw(.95,.6) -- (.55,1.4);
		\draw(1,.6) -- (1,1.4);
		\draw(1.05,.6) -- (1.45,1.4);
	\end{tikzpicture}
}
\newcommand{\sThirteen}[2]{
	\begin{tikzpicture}[rotate=#1,scale=#2]
		\nodea{1}{.5};
		\nodeb{1}{1};
		\nodeba{.5}{1.5};
		\nodebb{1}{1.5};
		\nodebc{1.5}{1.5};
		\draw(.9,1.1) -- (.6,1.4);
		\draw(1,1.1) -- (1,1.4);
		\draw(1.1,1.1) -- (1.4,1.4);
		\draw(.95,.6) -- (.55,1.4);
	\end{tikzpicture}
}
\newcommand{\sFourteen}[2]{
	\begin{tikzpicture}[rotate=#1,scale=#2]
		\nodea{1}{.5};
		\nodeb{1}{1};
		\nodeba{.5}{1.5};
		\nodebb{1}{1.5};
		\nodebc{1.5}{1.5};
		\draw(.9,1.1) -- (.6,1.4);
		\draw(1,1.1) -- (1,1.4);
		\draw(1.1,1.1) -- (1.4,1.4);
		\draw(.9,.6) .. controls (.75,.8) and (.75,1.2) .. (.9,1.4);
	\end{tikzpicture}
}
\newcommand{\sFifteen}[2]{
	\begin{tikzpicture}[rotate=#1,scale=#2]
		\nodea{1}{.5};
		\nodeb{1}{1};
		\nodeba{.5}{1.5};
		\nodebb{1}{1.5};
		\nodebc{1.5}{1.5};
		\draw(.9,1.1) -- (.6,1.4);
		\draw(1,1.1) -- (1,1.4);
		\draw(1.1,1.1) -- (1.4,1.4);
		\draw(1.05,.6) -- (1.45,1.4);
	\end{tikzpicture}
}
\newcommand{\sSixteen}[2]{
	\begin{tikzpicture}[rotate=#1,scale=#2]
		\nodea{1}{1};
		\nodeb{1}{.5};
		\nodeba{.5}{1.5};
		\nodebb{1}{1.5};
		\nodebc{1.5}{1.5};
		\draw(1,.6) -- (1,.9);
		\draw(.9,1.1) -- (.6,1.4);
		\draw(1,1.1) -- (1,1.4);
		\draw(1.05,.6) -- (1.45,1.4);
	\end{tikzpicture}
}
\newcommand{\sSeventeen}[2]{
	\begin{tikzpicture}[rotate=#1,scale=#2]
		\nodea{1}{1};
		\nodeb{1}{.5};
		\nodeba{.5}{1.5};
		\nodebb{1}{1.5};
		\nodebc{1.5}{1.5};
		\draw(1,.6) -- (1,.9);
		\draw(.9,1.1) -- (.6,1.4);
		\draw(1.1,1.1) -- (1.4,1.4);
		\draw(.9,.6) .. controls (.75,.8) and (.75,1.2) .. (.9,1.4);
	\end{tikzpicture}
}
\newcommand{\sEighteen}[2]{
	\begin{tikzpicture}[rotate=#1,scale=#2]
		\nodea{1}{1};
		\nodeb{1}{.5};
		\nodeba{.5}{1.5};
		\nodebb{1}{1.5};
		\nodebc{1.5}{1.5};
		\draw(1,.6) -- (1,.9);
		\draw(1.1,1.1) -- (1.4,1.4);
		\draw(1,1.1) -- (1,1.4);
		\draw(.95,.6) -- (.55,1.4);
	\end{tikzpicture}
}
\newcommand{\sNineteen}[2]{
	\begin{tikzpicture}[rotate=#1,scale=#2]
		\nodea{1}{.5};
		\nodeb{1}{1};
		\nodeba{.5}{1.5};
		\nodebb{1}{1.5};
		\nodebc{1.5}{1.5};
		\draw(.9,1.1) -- (.6,1.4);
		\draw(1,1.1) -- (1,1.4);
		\draw(1.1,1.1) -- (1.4,1.4);
		\draw(.95,.6) -- (.55,1.4);
		\draw(.9,.6) .. controls (.75,.8) and (.75,1.2) .. (.9,1.4);
	\end{tikzpicture}
}
\newcommand{\sTwenty}[2]{
	\begin{tikzpicture}[rotate=#1,scale=#2]
		\nodea{1}{.5};
		\nodeb{1}{1};
		\nodeba{.5}{1.5};
		\nodebb{1}{1.5};
		\nodebc{1.5}{1.5};
		\draw(.9,1.1) -- (.6,1.4);
		\draw(1,1.1) -- (1,1.4);
		\draw(1.1,1.1) -- (1.4,1.4);
		\draw(.95,.6) -- (.55,1.4);
		\draw(1.05,.6) -- (1.45,1.4);
	\end{tikzpicture}
}
\newcommand{\sTwentyone}[2]{
	\begin{tikzpicture}[rotate=#1,scale=#2]
		\nodea{1}{.5};
		\nodeb{1}{1};
		\nodeba{.5}{1.5};
		\nodebb{1}{1.5};
		\nodebc{1.5}{1.5};
		\draw(.9,1.1) -- (.6,1.4);
		\draw(1,1.1) -- (1,1.4);
		\draw(1.1,1.1) -- (1.4,1.4);
		\draw(.9,.6) .. controls (.75,.8) and (.75,1.2) .. (.9,1.4);
		\draw(1.05,.6) -- (1.45,1.4);
	\end{tikzpicture}
}
\newcommand{\sTwentytwo}[2]{
	\begin{tikzpicture}[rotate=#1,scale=#2]
		\nodea{1}{1};
		\nodeb{1}{.5};
		\nodeba{.5}{1.5};
		\nodebb{1}{1.5};
		\nodebc{1.5}{1.5};
		\draw(1,.6) -- (1,.9);
		\draw(1.1,1.1) -- (1.4,1.4);
		\draw(1,1.1) -- (1,1.4);
		\draw(.9,1.1) -- (.6,1.4);
	\end{tikzpicture}
}
\newcommand{\sTwentythree}[2]{
	\begin{tikzpicture}[rotate=#1,scale=#2]
		\nodea{1}{.5};
		\nodeb{1}{1};
		\nodeba{.5}{1.5};
		\nodebb{1}{1.5};
		\nodebc{1.5}{1.5};
		\draw(.9,1.1) -- (.6,1.4);
		\draw(1,1.1) -- (1,1.4);
		\draw(1.1,1.1) -- (1.4,1.4);
		\draw(.95,.6) -- (.55,1.4);
		\draw(.9,.6) .. controls (.75,.8) and (.75,1.2) .. (.9,1.4);
		\draw(1.05,.6) -- (1.45,1.4);
	\end{tikzpicture}
}
\newcommand{\sTwentyfour}[2]{
	\begin{tikzpicture}[rotate=#1,scale=#2]
		\nodea{1}{.5};
		\nodeb{1}{1};
		\nodeba{.5}{1.5};
		\nodebb{1}{1.5};
		\nodebc{1.5}{1.5};
		\draw(.9,1.1) -- (.6,1.4);
		\draw(1,1.1) -- (1,1.4);
		\draw(1.1,1.1) -- (1.4,1.4);
		\draw(1,.4) -- (1,.9);
	\end{tikzpicture}
}

\newcommand{\thOne}{\tiny
  \begin{tabular}{|c||c|c|c|c|}
	\hline
	$\hat{T}$ & $s_{0}$ & $s_{1}$ & $s_{2}$ & $s_{3}$ \\
	\hline\hline
	$c_{1}$ & $\times$ & & & \\ 
	\hline
  \end{tabular}
}
\newcommand{\thTwo}{\tiny
  \begin{tabular}{|c||c|c|c|c|}
	\hline
	$\hat{T}$ & $s_{0}$ & $s_{1}$ & $s_{2}$ & $s_{3}$ \\
	\hline\hline
	$c_{1}$ & $\times$ & $\times$ & & \\ 
	\hline
  \end{tabular}
}
\newcommand{\thThree}{\tiny
  \begin{tabular}{|c||c|c|c|c|}
	\hline
	$\hat{T}$ & $s_{0}$ & $s_{1}$ & $s_{2}$ & $s_{3}$ \\
	\hline\hline
	$c_{1}$ & $\times$ & & $\times$ & \\ 
	\hline
  \end{tabular}
}
\newcommand{\thFour}{\tiny
  \begin{tabular}{|c||c|c|c|c|}
	\hline
	$\hat{T}$ & $s_{0}$ & $s_{1}$ & $s_{2}$ & $s_{3}$ \\
	\hline\hline
	$c_{1}$ & $\times$ & & & $\times$ \\ 
	\hline
  \end{tabular}
}
\newcommand{\thFive}{\tiny
  \begin{tabular}{|c||c|c|c|c|}
	\hline
	$\hat{T}$ & $s_{0}$ & $s_{1}$ & $s_{2}$ & $s_{3}$ \\
	\hline\hline
	$c_{1}$ & $\times$ & $\times$ & $\times$ & \\ 
	\hline
  \end{tabular}
}
\newcommand{\thSix}{\tiny
  \begin{tabular}{|c||c|c|c|c|}
	\hline
	$\hat{T}$ & $s_{0}$ & $s_{1}$ & $s_{2}$ & $s_{3}$ \\
	\hline\hline
	$c_{1}$ & $\times$ & $\times$ & & $\times$ \\ 
	\hline
  \end{tabular}
}
\newcommand{\thSeven}{\tiny
  \begin{tabular}{|c||c|c|c|c|}
	\hline
	$\hat{T}$ & $s_{0}$ & $s_{1}$ & $s_{2}$ & $s_{3}$ \\
	\hline\hline
	$c_{1}$ & $\times$ & & $\times$ & $\times$ \\ 
	\hline
  \end{tabular}
}
\newcommand{\thEight}{\tiny
  \begin{tabular}{|c||c|c|c|c|}
	\hline
	$\hat{T}$ & $s_{0}$ & $s_{1}$ & $s_{2}$ & $s_{3}$ \\
	\hline\hline
	$c_{1}$ & $\times$ & $\times$ & $\times$ & $\times$ \\ 
	\hline
  \end{tabular}
}
\newcommand{\thNine}{\tiny
  \begin{tabular}{|c||c|c|c|c|}
	\hline
	$\hat{T}$ & $s_{0}$ & $s_{1}$ & $s_{2}$ & $s_{3}$ \\
	\hline\hline
	$c_{1}$ & & & & \\ 
	\hline
  \end{tabular}
}

\newcommand{\skPhi}{
	\def\x{.8};
	\def\y{.4};
	\draw(2*\x,1.5*\y) node[fill,circle,scale=.7](c1){};
	\draw(2*\x,4.5*\y) node[fill,circle,scale=.7](c2){};
	\draw(5*\x,1*\y) node[draw,circle,scale=.7](b1){};
	\draw(5*\x,2*\y) node[draw,circle,scale=.7](b2){};
	\draw(4*\x,3*\y) node[draw,circle,scale=.7](b3){};
	\draw(5*\x,3*\y) node[draw,circle,scale=.7](b4){};
	\draw(6*\x,3*\y) node[draw,circle,scale=.7](b5){};
	\draw(4*\x,4*\y) node[draw,circle,scale=.7](b6){};
	\draw(5*\x,4*\y) node[draw,circle,scale=.7](b7){};
	\draw(6*\x,4*\y) node[draw,circle,scale=.7](b8){};
	\draw(5*\x,5*\y) node[draw,circle,scale=.7](b9){};
	\draw(c1) -- (c2);
	\draw(b1) -- (b2);
	\draw(b2) -- (b3);
	\draw(b2) -- (b4);
	\draw(b2) -- (b5);
	\draw(b3) -- (b6);
	\draw(b3) -- (b7);
	\draw(b4) -- (b6);
	\draw(b4) -- (b8);
	\draw(b5) -- (b7);
	\draw(b5) -- (b8);
	\draw(b6) -- (b9);
	\draw(b7) -- (b9);
	\draw(b8) -- (b9);
}
\newcommand{\skPsi}{
	\def\x{.8};
	\def\y{.4};
	\draw(2*\x,1*\y) node[draw,circle,scale=.7](b1){};
	\draw(2*\x,2*\y) node[draw,circle,scale=.7](b2){};
	\draw(1*\x,3*\y) node[draw,circle,scale=.7](b3){};
	\draw(2*\x,3*\y) node[draw,circle,scale=.7](b4){};
	\draw(3*\x,3*\y) node[draw,circle,scale=.7](b5){};
	\draw(1*\x,4*\y) node[draw,circle,scale=.7](b6){};
	\draw(2*\x,4*\y) node[draw,circle,scale=.7](b7){};
	\draw(3*\x,4*\y) node[draw,circle,scale=.7](b8){};
	\draw(2*\x,5*\y) node[draw,circle,scale=.7](b9){};
	\draw(5*\x,1.5*\y) node[fill,circle,scale=.7](c1){};
	\draw(5*\x,4.5*\y) node[fill,circle,scale=.7](c2){};
	\draw(c1) -- (c2);
	\draw(b1) -- (b2);
	\draw(b2) -- (b3);
	\draw(b2) -- (b4);
	\draw(b2) -- (b5);
	\draw(b3) -- (b6);
	\draw(b3) -- (b7);
	\draw(b4) -- (b6);
	\draw(b4) -- (b8);
	\draw(b5) -- (b7);
	\draw(b5) -- (b8);
	\draw(b6) -- (b9);
	\draw(b7) -- (b9);
	\draw(b8) -- (b9);
}
\newcommand{\PhiOne}{
	\begin{tikzpicture}[>=stealth']\tiny
		\skPhi
		\draw[dashed,->,draw=black!50!white](c1) -- (b9);
		\draw[dashed,->,draw=black!50!white](c2) -- (b2);
	\end{tikzpicture}
}
\newcommand{\PhiTwo}{
	\begin{tikzpicture}[>=stealth']\tiny
		\skPhi
		\draw[dashed,->,draw=black!50!white](c1) -- (b9);
		\draw[dashed,->,draw=black!50!white](c2) -- (b3);
	\end{tikzpicture}
}
\newcommand{\PhiThree}{
	\begin{tikzpicture}[>=stealth']\tiny
		\skPhi
		\draw[dashed,->,draw=black!50!white](c1) -- (b9);
		\draw[dashed,->,draw=black!50!white](c2) -- (b4);
	\end{tikzpicture}
}
\newcommand{\PhiFour}{
	\begin{tikzpicture}[>=stealth']\tiny
		\skPhi
		\draw[dashed,->,draw=black!50!white](c1) -- (b9);
		\draw[dashed,->,draw=black!50!white](c2) -- (b5);
	\end{tikzpicture}
}
\newcommand{\PhiFive}{
	\begin{tikzpicture}[>=stealth']\tiny
		\skPhi
		\draw[dashed,->,draw=black!50!white](c1) -- (b9);
		\draw[dashed,->,draw=black!50!white](c2) -- (b6);
	\end{tikzpicture}
}
\newcommand{\PhiSix}{
	\begin{tikzpicture}[>=stealth']\tiny
		\skPhi
		\draw[dashed,->,draw=black!50!white](c1) -- (b9);
		\draw[dashed,->,draw=black!50!white](c2) -- (b7);
	\end{tikzpicture}
}
\newcommand{\PhiSeven}{
	\begin{tikzpicture}[>=stealth']\tiny
		\skPhi
		\draw[dashed,->,draw=black!50!white](c1) -- (b9);
		\draw[dashed,->,draw=black!50!white](c2) -- (b8);
	\end{tikzpicture}
}
\newcommand{\PhiEight}{
	\begin{tikzpicture}[>=stealth']\tiny
		\skPhi
		\draw[dashed,->,draw=black!50!white](c1) -- (b9);
		\draw[dashed,->,draw=black!50!white](c2) -- (b9);
	\end{tikzpicture}
}
\newcommand{\PhiNine}{
	\begin{tikzpicture}[>=stealth']\tiny
		\skPhi
		\draw[dashed,->,draw=black!50!white](c1) -- (b9);
		\draw[dashed,->,draw=black!50!white](c2) -- (b1);
	\end{tikzpicture}
}
\newcommand{\PsiOne}{
	\begin{tikzpicture}[>=stealth']\tiny
		\skPsi
		\draw[dashed,->,draw=black!50!white](b1) -- (c2);
		\draw[dashed,->,draw=black!50!white](b2) -- (c2);
		\draw[dashed,->,draw=black!50!white](b3) -- (c1);
		\draw[dashed,->,draw=black!50!white](b4) -- (c1);
		\draw[dashed,->,draw=black!50!white](b5) -- (c1);
		\draw[dashed,->,draw=black!50!white](b6) -- (c1);
		\draw[dashed,->,draw=black!50!white](b7) -- (c1);
		\draw[dashed,->,draw=black!50!white](b8) -- (c1);
		\draw[dashed,->,draw=black!50!white](b9) -- (c1);
	\end{tikzpicture}
}
\newcommand{\PsiTwo}{
	\begin{tikzpicture}[>=stealth']\tiny
		\skPsi
		\draw[dashed,->,draw=black!50!white](b1) -- (c2);
		\draw[dashed,->,draw=black!50!white](b2) -- (c2);
		\draw[dashed,->,draw=black!50!white](b3) -- (c2);
		\draw[dashed,->,draw=black!50!white](b4) -- (c1);
		\draw[dashed,->,draw=black!50!white](b5) -- (c1);
		\draw[dashed,->,draw=black!50!white](b6) -- (c1);
		\draw[dashed,->,draw=black!50!white](b7) -- (c1);
		\draw[dashed,->,draw=black!50!white](b8) -- (c1);
		\draw[dashed,->,draw=black!50!white](b9) -- (c1);
	\end{tikzpicture}
}
\newcommand{\PsiThree}{
	\begin{tikzpicture}[>=stealth']\tiny
		\skPsi
		\draw[dashed,->,draw=black!50!white](b1) -- (c2);
		\draw[dashed,->,draw=black!50!white](b2) -- (c2);
		\draw[dashed,->,draw=black!50!white](b3) -- (c1);
		\draw[dashed,->,draw=black!50!white](b4) -- (c2);
		\draw[dashed,->,draw=black!50!white](b5) -- (c1);
		\draw[dashed,->,draw=black!50!white](b6) -- (c1);
		\draw[dashed,->,draw=black!50!white](b7) -- (c1);
		\draw[dashed,->,draw=black!50!white](b8) -- (c1);
		\draw[dashed,->,draw=black!50!white](b9) -- (c1);
	\end{tikzpicture}
}
\newcommand{\PsiFour}{
	\begin{tikzpicture}[>=stealth']\tiny
		\skPsi
		\draw[dashed,->,draw=black!50!white](b1) -- (c2);
		\draw[dashed,->,draw=black!50!white](b2) -- (c2);
		\draw[dashed,->,draw=black!50!white](b3) -- (c1);
		\draw[dashed,->,draw=black!50!white](b4) -- (c1);
		\draw[dashed,->,draw=black!50!white](b5) -- (c2);
		\draw[dashed,->,draw=black!50!white](b6) -- (c1);
		\draw[dashed,->,draw=black!50!white](b7) -- (c1);
		\draw[dashed,->,draw=black!50!white](b8) -- (c1);
		\draw[dashed,->,draw=black!50!white](b9) -- (c1);
	\end{tikzpicture}
}
\newcommand{\PsiFive}{
	\begin{tikzpicture}[>=stealth']\tiny
		\skPsi
		\draw[dashed,->,draw=black!50!white](b1) -- (c2);
		\draw[dashed,->,draw=black!50!white](b2) -- (c2);
		\draw[dashed,->,draw=black!50!white](b3) -- (c2);
		\draw[dashed,->,draw=black!50!white](b4) -- (c2);
		\draw[dashed,->,draw=black!50!white](b5) -- (c1);
		\draw[dashed,->,draw=black!50!white](b6) -- (c2);
		\draw[dashed,->,draw=black!50!white](b7) -- (c1);
		\draw[dashed,->,draw=black!50!white](b8) -- (c1);
		\draw[dashed,->,draw=black!50!white](b9) -- (c1);
	\end{tikzpicture}
}
\newcommand{\PsiSix}{
	\begin{tikzpicture}[>=stealth']\tiny
		\skPsi
		\draw[dashed,->,draw=black!50!white](b1) -- (c2);
		\draw[dashed,->,draw=black!50!white](b2) -- (c2);
		\draw[dashed,->,draw=black!50!white](b3) -- (c2);
		\draw[dashed,->,draw=black!50!white](b4) -- (c1);
		\draw[dashed,->,draw=black!50!white](b5) -- (c2);
		\draw[dashed,->,draw=black!50!white](b6) -- (c1);
		\draw[dashed,->,draw=black!50!white](b7) -- (c2);
		\draw[dashed,->,draw=black!50!white](b8) -- (c1);
		\draw[dashed,->,draw=black!50!white](b9) -- (c1);
	\end{tikzpicture}
}
\newcommand{\PsiSeven}{
	\begin{tikzpicture}[>=stealth']\tiny
		\skPsi
		\draw[dashed,->,draw=black!50!white](b1) -- (c2);
		\draw[dashed,->,draw=black!50!white](b2) -- (c2);
		\draw[dashed,->,draw=black!50!white](b3) -- (c1);
		\draw[dashed,->,draw=black!50!white](b4) -- (c2);
		\draw[dashed,->,draw=black!50!white](b5) -- (c2);
		\draw[dashed,->,draw=black!50!white](b6) -- (c1);
		\draw[dashed,->,draw=black!50!white](b7) -- (c1);
		\draw[dashed,->,draw=black!50!white](b8) -- (c2);
		\draw[dashed,->,draw=black!50!white](b9) -- (c1);
	\end{tikzpicture}
}
\newcommand{\PsiEight}{
	\begin{tikzpicture}[>=stealth']\tiny
		\skPsi
		\draw[dashed,->,draw=black!50!white](b1) -- (c2);
		\draw[dashed,->,draw=black!50!white](b2) -- (c2);
		\draw[dashed,->,draw=black!50!white](b3) -- (c2);
		\draw[dashed,->,draw=black!50!white](b4) -- (c2);
		\draw[dashed,->,draw=black!50!white](b5) -- (c2);
		\draw[dashed,->,draw=black!50!white](b6) -- (c2);
		\draw[dashed,->,draw=black!50!white](b7) -- (c2);
		\draw[dashed,->,draw=black!50!white](b8) -- (c2);
		\draw[dashed,->,draw=black!50!white](b9) -- (c2);
	\end{tikzpicture}
}
\newcommand{\PsiNine}{
	\begin{tikzpicture}[>=stealth']\tiny
		\skPsi
		\draw[dashed,->,draw=black!50!white](b1) -- (c2);
		\draw[dashed,->,draw=black!50!white](b2) -- (c1);
		\draw[dashed,->,draw=black!50!white](b3) -- (c1);
		\draw[dashed,->,draw=black!50!white](b4) -- (c1);
		\draw[dashed,->,draw=black!50!white](b5) -- (c1);
		\draw[dashed,->,draw=black!50!white](b6) -- (c1);
		\draw[dashed,->,draw=black!50!white](b7) -- (c1);
		\draw[dashed,->,draw=black!50!white](b8) -- (c1);
		\draw[dashed,->,draw=black!50!white](b9) -- (c1);
	\end{tikzpicture}
}

\newcommand{\skZeta}{
	\def\x{.8};
	\def\y{.5};
	\draw(1*\x,4*\y) node[fill,circle,scale=.7,label=left:$c_{1}$](c1){};
	\draw(3*\x,1*\y) node(b1){$x$};
	\draw(3*\x,2*\y) node(b2){$\emptyset$};
	\draw(2*\x,3*\y) node(b3){$0$};
	\draw(3*\x,3*\y) node(b4){$1$};
	\draw(4*\x,3*\y) node(b5){$2$};
	\draw(2*\x,4*\y) node(b6){$01$};
	\draw(3*\x,4*\y) node(b7){$02$};
	\draw(4*\x,4*\y) node(b8){$12$};
	\draw(3*\x,5*\y) node(b9){$012$};
	\draw(2.5*\x,6*\y) node(b10){$1'$};
	\draw(3.5*\x,6*\y) node(b11){$2'$};
	\draw(3*\x,7*\y) node(b12){$1'2'$};
	\draw(b1) -- (b2);
	\draw(b2) -- (b3);
	\draw(b2) -- (b4);
	\draw(b2) -- (b5);
	\draw(b3) -- (b6);
	\draw(b3) -- (b7);
	\draw(b4) -- (b6);
	\draw(b4) -- (b8);
	\draw(b5) -- (b7);
	\draw(b5) -- (b8);
	\draw(b6) -- (b9);
	\draw(b7) -- (b9);
	\draw(b8) -- (b9);
	\draw(b9) -- (b10);
	\draw(b9) -- (b11);
	\draw(b10) -- (b12);
	\draw(b11) -- (b12);
}

\newcommand{\ZetaOne}{
	\begin{tikzpicture}[>=stealth']\tiny
		\skZeta
		\draw[dashed,->,draw=black!50!white](c1) -- (b1);
	\end{tikzpicture}
}

\newcommand{\ZetaTwo}{
	\begin{tikzpicture}[>=stealth']\tiny
		\skZeta
		\draw[dashed,->,draw=black!50!white](c1) .. controls (1.6*\x,2.5*\y) and (2.2*\x,2*\y) .. (b2);
	\end{tikzpicture}
}

\newcommand{\ZetaThree}{
	\begin{tikzpicture}[>=stealth']\tiny
		\skZeta
		\draw[dashed,->,draw=black!50!white](c1) -- (b3);
	\end{tikzpicture}
}

\newcommand{\ZetaFour}{
	\begin{tikzpicture}[>=stealth']\tiny
		\skZeta
		\draw[dashed,->,draw=black!50!white](c1) -- (b4);
	\end{tikzpicture}
}

\newcommand{\ZetaFive}{
	\begin{tikzpicture}[>=stealth']\tiny
		\skZeta
		\draw[dashed,->,draw=black!50!white](c1) -- (b5);
	\end{tikzpicture}
}

\newcommand{\ZetaSix}{
	\begin{tikzpicture}[>=stealth']\tiny
		\skZeta
		\draw[dashed,->,draw=black!50!white](c1) -- (b6);
	\end{tikzpicture}
}

\newcommand{\ZetaSeven}{
	\begin{tikzpicture}[>=stealth']\tiny
		\skZeta
		\draw[dashed,->,draw=black!50!white](c1) .. controls (1.5*\x,4.5*\y) and (2.5*\x,4.5*\y) .. (b7);
	\end{tikzpicture}
}

\newcommand{\ZetaEight}{
	\begin{tikzpicture}[>=stealth']\tiny
		\skZeta
		\draw[dashed,->,draw=black!50!white](c1) .. controls (1.5*\x,4.6*\y) and (3.5*\x,4.6*\y) .. (b8);
	\end{tikzpicture}
}

\newcommand{\ZetaNine}{
	\begin{tikzpicture}[>=stealth']\tiny
		\skZeta
		\draw[dashed,->,draw=black!50!white](c1) -- (b9);
	\end{tikzpicture}
}

\newcommand{\ZetaTen}{
	\begin{tikzpicture}[>=stealth']\tiny
		\skZeta
		\draw[dashed,->,draw=black!50!white](c1) -- (b10);
	\end{tikzpicture}
}

\newcommand{\ZetaEleven}{
	\begin{tikzpicture}[>=stealth']\tiny
		\skZeta
		\draw[dashed,->,draw=black!50!white](c1) -- (b11);
	\end{tikzpicture}
}

\newcommand{\ZetaTwelve}{
	\begin{tikzpicture}[>=stealth']\tiny
		\skZeta
		\draw[dashed,->,draw=black!50!white](c1) .. controls (1.3*\x,5*\y) and (2*\x,6*\y) .. (b12);
	\end{tikzpicture}
}

\newcommand{\rzOne}{\tiny
	\begin{tabular}{|c||c|c|c|}
		\hline
		$R_{\zeta}$ & $s_{0}$ & $s_{1}$ & $s_{2}$\\
		\hline\hline
		$c_{1}$ & $\times$ & $\times$ & $\times$ \\ 
		\hline
	\end{tabular}  
}

\newcommand{\rzTwo}{\tiny
	\begin{tabular}{|c||c|c|c|}
		\hline
		$R_{\zeta}$ & $s_{0}$ & $s_{1}$ & $s_{2}$\\
		\hline\hline
		$c_{1}$ & & & \\ 
		\hline
	\end{tabular}  
}

\newcommand{\rzThree}{\tiny
	\begin{tabular}{|c||c|c|c|}
		\hline
		$R_{\zeta}$ & $s_{0}$ & $s_{1}$ & $s_{2}$\\
		\hline\hline
		$c_{1}$ & & $\times$ & $\times$ \\ 
		\hline
	\end{tabular}  
}

\newcommand{\rzFour}{\tiny
	\begin{tabular}{|c||c|c|c|}
		\hline
		$R_{\zeta}$ & $s_{0}$ & $s_{1}$ & $s_{2}$\\
		\hline\hline
		$c_{1}$ & & & $\times$ \\ 
		\hline
	\end{tabular}  
}

\newcommand{\rzFive}{\tiny
	\begin{tabular}{|c||c|c|c|}
		\hline
		$R_{\zeta}$ & $s_{0}$ & $s_{1}$ & $s_{2}$\\
		\hline\hline
		$c_{1}$ & & $\times$ & \\ 
		\hline
	\end{tabular}  
}

\newcommand{\tzOne}{\tiny
	\begin{tabular}{|c||c|}
		\hline
		$T_{\zeta}$ & $c_{1}$\\
		\hline\hline
		$s_{0}$ & \\ 
		\hline
		$s_{1}$ & \\ 
		\hline
		$s_{2}$ & \\ 
		\hline
	\end{tabular}  
}

\newcommand{\tzTwo}{\tiny
	\begin{tabular}{|c||c|}
		\hline
		$T_{\zeta}$ & $c_{1}$\\
		\hline\hline
		$s_{0}$ & $\times$ \\ 
		\hline
		$s_{1}$ & \\ 
		\hline
		$s_{2}$ & \\ 
		\hline
	\end{tabular}  
}

\newcommand{\tzThree}{\tiny
	\begin{tabular}{|c||c|}
		\hline
		$T_{\zeta}$ & $c_{1}$\\
		\hline\hline
		$s_{0}$ & $\times$ \\ 
		\hline
		$s_{1}$ & $\times$ \\ 
		\hline
		$s_{2}$ & \\ 
		\hline
	\end{tabular}  
}

\newcommand{\tzFour}{\tiny
	\begin{tabular}{|c||c|}
		\hline
		$T_{\zeta}$ & $c_{1}$\\
		\hline\hline
		$s_{0}$ & $\times$ \\ 
		\hline
		$s_{1}$ & \\ 
		\hline
		$s_{2}$ & $\times$ \\ 
		\hline
	\end{tabular}  
}

\newcommand{\tzFive}{\tiny
	\begin{tabular}{|c||c|}
		\hline
		$T_{\zeta}$ & $c_{1}$\\
		\hline\hline
		$s_{0}$ & $\times$ \\ 
		\hline
		$s_{1}$ & $\times$ \\ 
		\hline
		$s_{2}$ & $\times$ \\ 
		\hline
	\end{tabular}  
}

\newcommand{\pzOne}[2]{
	\begin{tikzpicture}[rotate=#1,scale=#2]
		\nodea{1}{.5};
		\nodeb{1}{1};
		\nodeba{.5}{1.5};
		\nodebb{1.5}{1.5};
		\draw(1,.6) -- (1,.9);
		\draw(.9,1.1) -- (.6,1.4);
		\draw(1.1,1.1) -- (1.4,1.4);
	\end{tikzpicture}
}

\newcommand{\pzTwo}[2]{
	\begin{tikzpicture}[rotate=#1,scale=#2]
		\nodea{1}{1};
		\nodeb{1}{.5};
		\nodeba{.5}{1.5};
		\nodebb{1.5}{1.5};
		\draw(.95,.6) -- (.55,1.4);
		\draw(1.05,.6) -- (1.45,1.4);
		\draw(.9,1.1) -- (.6,1.4);
		\draw(1.1,1.1) -- (1.4,1.4);
	\end{tikzpicture}
}

\newcommand{\pzThree}[2]{
	\begin{tikzpicture}[rotate=#1,scale=#2]
		\nodea{1}{1};
		\nodeb{1}{.5};
		\nodeba{.5}{1.5};
		\nodebb{1.5}{1.5};
		\draw(1,.6) -- (1,.9);
		\draw(.9,1.1) -- (.6,1.4);
		\draw(1.1,1.1) -- (1.4,1.4);
	\end{tikzpicture}
}

\newcommand{\pzFour}[2]{
	\begin{tikzpicture}[rotate=#1,scale=#2]
		\nodea{1}{1};
		\nodeb{1}{.5};
		\nodeba{.5}{1.5};
		\nodebb{1.5}{1.5};
		\draw(.95,.6) -- (.55,1.4);
		\draw(1.05,.6) -- (1.45,1.4);
		\draw(1.1,1.1) -- (1.4,1.4);
	\end{tikzpicture}
}

\newcommand{\pzFive}[2]{
	\begin{tikzpicture}[rotate=#1,scale=#2]
		\nodea{1}{1};
		\nodeb{1}{.5};
		\nodeba{.5}{1.5};
		\nodebb{1.5}{1.5};
		\draw(.95,.6) -- (.55,1.4);
		\draw(1.05,.6) -- (1.45,1.4);
		\draw(.9,1.1) -- (.6,1.4);
	\end{tikzpicture}
}

\newcommand{\pzSix}[2]{
	\begin{tikzpicture}[rotate=#1,scale=#2]
		\nodea{1.25}{1};
		\nodeb{1}{.5};
		\nodeba{.5}{1.5};
		\nodebb{1.5}{1.5};
		\draw(.95,.6) -- (.55,1.4);
		\draw(1.05,.6) -- (1.45,1.4);
	\end{tikzpicture}
}

\newcommand{\pzSeven}[2]{
	\begin{tikzpicture}[rotate=#1,scale=#2]
		\nodea{.75}{1};
		\nodeb{1}{.5};
		\nodeba{.5}{1.5};
		\nodebb{1.5}{1.5};
		\draw(.95,.6) -- (.55,1.4);
		\draw(1.05,.6) -- (1.45,1.4);
	\end{tikzpicture}
}

\newcommand{\pzEight}[2]{
	\begin{tikzpicture}[rotate=#1,scale=#2]
		\nodea{1}{1};
		\nodeb{1}{.5};
		\nodeba{.5}{1.5};
		\nodebb{1.5}{1.5};
		\draw(.95,.6) -- (.55,1.4);
		\draw(1.05,.6) -- (1.45,1.4);
	\end{tikzpicture}
}

\newcommand{\pzNine}[2]{
	\begin{tikzpicture}[rotate=#1,scale=#2]
		\nodea{1}{1};
		\nodeb{1}{.5};
		\nodeba{.5}{1.5};
		\nodebb{1.5}{1.5};
		\draw(.95,.6) -- (.55,1.4);
		\draw(1.05,.6) -- (1.45,1.4);
		\draw(1,.6) -- (1,.9);
	\end{tikzpicture}
}

\newcommand{\pzTen}[2]{
	\begin{tikzpicture}[rotate=#1,scale=#2]
		\nodea{1}{1.5};
		\nodeb{1}{.5};
		\nodeba{.5}{1};
		\nodebb{1.5}{1};
		\draw(.9,.6) -- (.6,.9);
		\draw(1.1,.6) -- (1.4,.9);
		\draw(.6,1.1) -- (.9,1.4);
	\end{tikzpicture}
}

\newcommand{\pzEleven}[2]{
	\begin{tikzpicture}[rotate=#1,scale=#2]
		\nodea{1}{1.5};
		\nodeb{1}{.5};
		\nodeba{.5}{1};
		\nodebb{1.5}{1};
		\draw(.9,.6) -- (.6,.9);
		\draw(1.1,.6) -- (1.4,.9);
		\draw(1.4,1.1) -- (1.1,1.4);
	\end{tikzpicture}
}

\newcommand{\pzTwelve}[2]{
	\begin{tikzpicture}[rotate=#1,scale=#2]
		\nodea{1}{1.5};
		\nodeb{1}{.5};
		\nodeba{.5}{1};
		\nodebb{1.5}{1};
		\draw(.9,.6) -- (.6,.9);
		\draw(1.1,.6) -- (1.4,.9);
		\draw(.6,1.1) -- (.9,1.4);
		\draw(1.4,1.1) -- (1.1,1.4);
	\end{tikzpicture}
}

\pgfdeclarelayer{background}
\pgfsetlayers{background,main}
\newtheorem{theorem}{Theorem}[section]
\newtheorem{proposition}[theorem]{Proposition}

\newtheorem{lemma}[theorem]{Lemma}

\newtheorem{problem}[theorem]{Problem}
\theoremstyle{remark}
\newtheorem{remark}[theorem]{Remark}
\newcommand{\alert}[1]{{\color{DarkGreen}\emph{#1}}}
\newcommand{\parr}{\leftarrow_{P}}
\newcommand{\qarr}{\leftarrow_{Q}}
\newcommand{\fs}{\mathfrak{s}}
\newcommand{\cc}{\mathfrak{c}}
\newcommand{\ac}{\mathfrak{a}}
\newcommand{\KK}{\mathbb{K}}

\newcommand{\PP}{\mathcal{P}}
\newcommand{\QQ}{\mathcal{Q}}
\newcommand{\BB}{\mathcal{B}}
\newcommand{\CL}{\underline{\mathfrak{B}}}
\newcommand\rBrace[2]{%
	\left.\rule{0pt}{#1}\right\}\text{#2}}
\author{Henri M\"uhle
}
\title[Counting Proper Mergings of Stars and Chains]{
  Proper Mergings of Stars and Chains are Counted by Sums of Antidiagonals in Certain 
  Convolution Arrays\\ -- The Details --
}
\address{Fakult\"at f\"ur Mathematik, Universit\"at Wien, Vienna, Austria}
\email{henri.muehle@univie.ac.at}
\thanks{Supported by the FWF research grant no. Z130-N13.}

\begin{document}

\allowdisplaybreaks

\begin{abstract}
	A proper merging of two disjoint quasi-ordered sets $P$ and $Q$ is a quasi-order on the 
	union of $P$ and $Q$ such that the restriction to $P$ or $Q$ yields the original 
	quasi-order again and such that no elements of $P$ and $Q$ are identified. In this article, 
	we determine the number of proper mergings in the case where $P$ is a star (\emph{i.e.} an antichain 
	with a smallest element adjoined), and $Q$ is a chain. We show that the lattice of proper mergings
	of an $m$-antichain and an $n$-chain, previously investigated by the author, is a quotient lattice
	of the lattice of proper mergings of an $m$-star and an $n$-chain, and we determine the number of
	proper mergings of an $m$-star and an $n$-chain by counting the number of congruence classes and
	by determining their cardinalities. Additionally, we compute the number of 
	Galois connections between certain modified Boolean lattices and chains.
\end{abstract}

\maketitle

\section{Introduction}
  \label{sec:introduction}
Given two quasi-ordered sets $(P,\parr)$ and $(Q,\qarr)$, a merging of $P$ and $Q$ is a quasi-order 
$\leftarrow$ on the union of $P$ and $Q$ such that the restriction of $\leftarrow$ to $P$ 
or $Q$ yields $\parr$ respectively $\qarr$ again. In other words, a merging of $P$ 
and $Q$ is a quasi-order on the union of $P$ and $Q$, which does not change the quasi-orders on $P$ and 
$Q$. 

In \cite{ganter11merging} a characterization of the set of mergings of two arbitrary quasi-ordered sets 
$P$ and $Q$ is given. In particular, it turns out that every merging $\leftarrow$ of $P$ and $Q$ can be 
uniquely described by two binary relations $R\subseteq P\times Q$ and $T\subseteq Q\times P$. The 
relation $R$ can be interpreted as a description, which part of $P$ is weakly below $Q$, and analogously 
the relation $T$ can be interpreted as a description, which part of $Q$ is weakly below $P$. It was shown
in \cite{ganter11merging} that the set of mergings forms a distributive lattice in a 
natural way. If a merging satisfies $R\cap T^{-1}=\emptyset$, and hence if no element of $P$ is identified 
with an element of $Q$, then it is called proper, and the set of proper mergings forms a distributive 
sublattice of the previous one.

In \cite{muehle12counting}, the author gave formulas for the number of proper mergings of (i) an $m$-chain 
and an $n$-chain, (ii) an $m$-antichain and an $n$-antichain and (iii) an $m$-antichain and an $n$-chain,
see \cite{muehle12counting}*{Theorem~1.1}. The present article can be seen as a subsequent work which was 
triggered by the following observation: if we denote the number of proper mergings of an $m$-star 
(\emph{i.e.} an $m$-antichain with a minimal element adjoined) and an $n$-chain by $F_{\fs\!\cc}(m,n)$,
then the first few entries of $F_{\fs\!\cc}(2,n)$ (starting with $n=0$) are
\begin{displaymath}
	1, 12, 68, 260, 777, 1960, 4368, \ldots,
\end{displaymath}
and the first few entries of $F_{\fs\!\cc}(3,n)$ (starting with $n=0$) are
\begin{displaymath}
	1, 24, 236, 1400, 6009, 20608, 59952, \ldots.
\end{displaymath}
Surprisingly, these sequences are \cite{sloane}*{A213547} and \cite{sloane}*{A213560}, respectively, and
they describe sums of antidiagonals in certain convolution arrays. Inspired by this connection, we are
able to prove the following theorem.

\begin{theorem}
  \label{thm:cardinality}
	Let $\mathfrak{S\!C}_{m,n}^{\bullet}$ denote the set of proper mergings of an $m$-star and an 
	$n$-chain. Then,
	\begin{displaymath}
		\Bigl\lvert\mathfrak{S\!C}_{m,n}^{\bullet}\Bigr\rvert = 
		  \sum_{k=1}^{n+1}{k^{m}(n-k+2)^{m+1}}.
	\end{displaymath}
\end{theorem}

The proof of Theorem~\ref{thm:cardinality} is obtained in the following way: after recalling the necessary
notations and definitions in Section~\ref{sec:preliminaries}, we observe in Section~\ref{sec:embedding}
that the lattice $\bigl(\mathfrak{S\!C}_{m,n}^{\bullet},\preceq\bigr)$ contains a certain quotient lattice, 
namely the lattice $\bigl(\mathfrak{A\!C}_{m,n}^{\bullet},\preceq\bigr)$ of proper mergings of an 
$m$-antichain and an $n$-chain. The cardinality of $\mathfrak{A\!C}_{m,n}^{\bullet}$ was determined by the 
author in \cite{muehle12counting}. Then, in Section~\ref{sec:enumeration}, we determine the cardinalities
of the congruence classes of the lattice congruence generating 
$\bigl(\mathfrak{A\!C}_{m,n}^{\bullet},\preceq\bigr)$ as a quotient lattice of 
$\bigl(\mathfrak{S\!C}_{m,n}^{\bullet},\preceq\bigr)$, using a decomposition of 
$\mathfrak{A\!C}_{m,n}^{\bullet}$ by means of the bijection with monotone $(n+1)$-colorings of the complete 
bipartite digraph $\vec{K}_{m,m}$ described in \cite{muehle12counting}*{Section~5}. 

Using a theorem from Formal Concept Analysis which relates Galois connections between lattices to binary
relations between their formal contexts, we are able to determine the number of Galois connections between
certain modified Boolean lattices and chains in Section~\ref{sec:galois}. The mentioned modified Boolean
lattices and chains arise in a natural way, when considering proper mergings of stars and chains, thus we have
decided to include this result in the present article.

\section{Preliminaries}
  \label{sec:preliminaries}
In this section we recall the basic notations and definitions needed in this article. For a
detailed introduction to Formal Concept Analysis, we refer to \cite{ganter99formal}. 

\subsection{Formal Concept Analysis}
  \label{sec:formal_concept_analysis}
The theory of Formal Concept Ana\-lysis (FCA) was introduced in the 1980s by Rudolf Wille,
see \cite{wille82restructuring}, as an approach to restructure lattice theory. The initial goal was 
to interpret lattices as hierarchies of concepts and thus to give meaning to the lattice elements in 
a fixed context. Such a \alert{formal context} is a triple $(G,M,I)$, where $G$ is a set of 
so-called \alert{objects}, $M$ is a set of so-called \alert{attributes} and $I\subseteq G\times M$ is
a binary relation that describes whether an object \alert{has} an attribute. Given a formal context
$\mathbb{K}=(G,M,I)$, we define two derivation operators
\begin{align*}
	(\cdot)^{I}:\wp(G)\rightarrow\wp(M), & 
	  \quad A\mapsto A^{I}=\{m\in M\mid g\;I\;m\;\mbox{for all}\;g\in A\},\\
	(\cdot)^{I}:\wp(M)\rightarrow\wp(G), & 
	  \quad B\mapsto B^{I}=\{g\in G\mid g\;I\;m\;\mbox{for all}\;m\in B\},
\end{align*}
where $\wp$ denotes the power set. The notation $g\;I\;m$ is to be understood as $(g,m)\in I$. 
Let now $A\subseteq G$, and $B\subseteq M$. For brevity, if $g\in G$, then we write simply
$g^{I}$ instead of $\{g\}^{I}$, and analogously if $m\in M$, then we write $m^{I}$ instead of $\{m\}^{I}$. 
The pair $\mathfrak{b}=(A,B)$ is called \alert{formal concept of $\KK$} if $A^{I}=B$ and $B^{I}=A$. In this 
case, we call $A$ the \alert{extent} and $B$ the \alert{intent of $\mathfrak{b}$}. It can easily be seen that 
for every $A\subseteq G$, and $B\subseteq M$, the pairs $\bigl(A^{II},A^{I}\bigr)$ and respectively
$\bigl(B^{I},B^{II}\bigr)$ are formal concepts. Conversely, every formal concept of 
$\KK$ can be written in such a way. We denote the set of all formal concepts of 
$\KK$ by $\mathfrak{B}(\KK)$, and define a partial order on $\mathfrak{B}(\KK)$ by
\begin{displaymath}
	(A_{1},B_{1})\leq(A_{2},B_{2})\quad\mbox{if and only if}\quad
	  A_{1}\subseteq A_{2}\quad(\mbox{or equivalently}\;B_{1}\supseteq B_{2}).
\end{displaymath}
Let $\CL(\KK)$ denote the poset $\bigl(\mathfrak{B}(\KK),\leq\bigr)$. The basic theorem of FCA (see 
\cite{ganter99formal}*{Theorem~3}) states that $\CL(\KK)$ is a complete lattice, the so-called 
\alert{concept lattice of $\KK$}. Moreover, every complete lattice is a concept lattice. 

Usually, a formal context is represented by a cross-table, where the rows represent the 
objects and the columns represent the attributes. The cell in row $g$ and column $m$ contains a 
cross if and only if $g\;I\;m$. For every context $\KK=(G,M,I)$, there are two maps
\begin{align}\label{eq:maps}\begin{aligned}
	\gamma & :G\rightarrow\CL(\KK),\quad && g\mapsto \bigl(g^{II},g^{I}\bigr),
	  \quad\mbox{and}\\
	\mu & :M\rightarrow\CL(\KK),\quad && m\mapsto \bigl(m^{I},m^{II}\bigr),
\end{aligned}\end{align}
mapping each object, respectively attribute, to its corresponding formal concept. It is common sense
in FCA to label the Hasse diagram of $\CL(\KK)$ in the following way: the node representing a formal 
concept $\mathfrak{b}\in\mathfrak{B}(\KK)$ is labeled with the object $g$ (or with the attribute $m$) 
if and only if $\mathfrak{b}=\gamma g$ (or $\mathfrak{b}=\mu m$). Object labels are attached 
below the nodes in the Hasse diagram, and attribute labels above. In this presentation, the extent 
(intent) of a formal concept corresponds to the labels weakly below (weakly above) this formal 
concept in the Hasse diagram of $\CL(\KK)$. See Figures~\ref{fig:4_star} and \ref{fig:4_chain} for small 
examples. 

\subsection{Bonds and Mergings}
  \label{sec:bonds_mergings}
Let $\KK_{1}=(G_{1},M_{1},I_{1})$, and $\KK_{2}=(G_{2},M_{2},I_{2})$ be formal contexts. A binary
relation $R\subseteq G_{1}\times M_{2}$ is called \alert{bond from $\KK_{1}$ to $\KK_{2}$} if for 
every object $g\in G_{1}$, the row $g^{R}$ is an intent of $\KK_{2}$ and for every 
$m\in M_{2}$, the column $m^{R}$ is an extent of $\KK_{1}$. 

Now let $(P,\parr)$ and $(Q,\qarr)$ be disjoint quasi-ordered sets. Let $R\subseteq P\times Q$, and 
$T\subseteq Q\times P$. Define a relation $\leftarrow_{R,T}$ on $P\cup Q$ as
\begin{equation}
  \label{eq:merging}
	p\leftarrow_{R,T} q\quad\mbox{if and only if}\quad
	  p\parr q\;\;\text{or}\;\;p\qarr q\;\;\text{or}\;\;p\;R\;q\;\;\text{or}\;\;p\;T\;q,
\end{equation}
for all $p,q\in P\cup Q$. The pair $(R,T)$ is called \alert{merging of $P$ and $Q$} if 
$(P\cup Q,\leftarrow_{R,T})$ is a quasi-ordered set. Moreover, a merging is called \alert{proper} if 
$R\cap T^{-1}=\emptyset$. Since for fixed quasi-ordered sets $(P,\parr)$ and $(Q,\qarr)$ the relation 
$\leftarrow_{R,T}$ is uniquely determined by $R$ and $T$, we refer to $\leftarrow_{R,T}$ as a 
(proper) merging of $P$ and $Q$ as well. Let $\circ$ denote the relational 
product.

\begin{proposition}[\cite{ganter11merging}*{Proposition~2}]
  \label{prop:classification_mergings}
	Let $(P,\parr)$ and $(Q,\qarr)$ be disjoint quasi-ordered sets, and let 
	$R\subseteq P\times Q$, and $T\subseteq Q\times P$. The pair $(R,T)$ is a merging of $P$ and $Q$ 
	if and only if all of the following properties are satisfied:
	\begin{enumerate}
		\item $R$ is a bond from $(P,P,\not\rightarrow_{P})$ to $(Q,Q,\not\rightarrow_{Q})$,
		\item $T$ is a bond from $(Q,Q,\not\rightarrow_{Q})$ to $(P,P,\not\rightarrow_{P})$,
		\item $R\circ T$ is contained in $\parr$, and
		\item $T\circ R$ is contained in $\qarr$.
	\end{enumerate}
	Moreover, the relation $\leftarrow_{R,T}$ as defined in \eqref{eq:merging} is antisymmetric if 
	and only if $\parr$ and $\qarr$ are both antisymmetric and $R\cap T^{-1}=\emptyset$. 
\end{proposition}
In the case that $P$ and $Q$ are posets, this proposition implies that $(P\cup Q,\leftarrow_{R,T})$ is 
a poset again if and only if $(R,T)$ is a proper merging of $P$ and $Q$. 
Denote the set of mergings of $P$ and $Q$ by 
$\mathfrak{M}_{P,Q}$, and define a partial order on $\mathfrak{M}_{P,Q}$ by 
\begin{equation}
  \label{eq:merging_order}
	(R_{1},T_{1})\preceq(R_{2},T_{2})\quad\mbox{if and only if}\quad 
	  R_{1}\subseteq R_{2}\;\mbox{and}\;T_{1}\supseteq T_{2}.
\end{equation}
It was shown in \cite{ganter11merging}*{Theorem~1} that $\bigl(\mathfrak{M}_{P,Q},\preceq\bigr)$ is a
lattice, where $(\emptyset,Q\times P)$ is the unique minimal element, and $(P\times Q,\emptyset)$ the
unique maximal element. Moreover, it follows from \cite{ganter11merging}*{Theorem~2} that 
$\bigl(\mathfrak{M}_{P,Q},\preceq\bigr)$ is distributive. Let 
$\mathfrak{M}_{P,Q}^{\bullet}\subseteq\mathfrak{M}_{P,Q}$ denote the set of all \alert{proper} mergings of 
$P$ and $Q$. It was also shown in \cite{ganter11merging} that 
$\bigl(\mathfrak{M}_{P,Q}^{\bullet},\preceq\bigr)$ is a distributive sublattice of 
$\bigl(\mathfrak{M}_{P,Q},\preceq\bigr)$. 

\subsection{$m$-Stars}
  \label{sec:m_stars}
Let $A=\{a_{1},a_{2},\ldots,a_{m}\}$ be a set. An $m$-antichain is a poset $\ac=(A,=_{\ac})$, satisfying 
$a_{i}=_{\ac}a_{j}$ if and only if $i=j$ for all $i,j\in\{1,2,\ldots,m\}$. Consider the set $S=A\cup\{s_{0}\}$,
and define a partial order $\leq_{\fs}$ on $S$ as follows: $s\leq_{\fs} s'$ if and only if either $s=s'$ or
$s=s_{0}$ for all $s,s'\in S$. The poset $\fs=(S,\leq_{\fs})$ is called an \alert{$m$-star}. (That is, an 
$m$-star is an $m$-antichain with a smallest element adjoined. See Figure~\ref{fig:4_star} for an example.)
We are interested in the formal concepts of the contraordinal scale of an $m$-star, namely the formal
concepts of the formal context $(S,S,\not\geq_{\fs})$. It is clear that 
$\bigl(\emptyset,S\bigr)$ is a formal concept of $(S,S,\not\geq_{\fs})$, and we notice further that, for 
every $B\subseteq S\setminus\{s_{0}\}$ (considered as an object set), we have 
$B^{\not\geq_{\fs}}=S\setminus\bigl(B\cup\{s_{0}\}\bigr)$. Since the object $s_{0}$ satisfies
$s_{0}^{\not\geq_{\fs}}=S\setminus\{s_{0}\}$, we conclude further that
$B^{\not\geq_{\fs}\not\geq_{\fs}}=B\cup\{s_{0}\}$. Thus, $(S,S,\not\geq_{\fs})$ has precisely $2^{m}+1$ 
formal concepts, namely 
\begin{displaymath}
	\bigl(\emptyset,S\bigr)\quad\mbox{and}
	  \quad\Bigl(B\cup\{s_{0}\},S\setminus\bigl(B\cup\{s_{0}\}\bigr)\Bigr)\quad
	  \mbox{for}\;B\subseteq S\setminus\{s_{0}\}.
\end{displaymath}

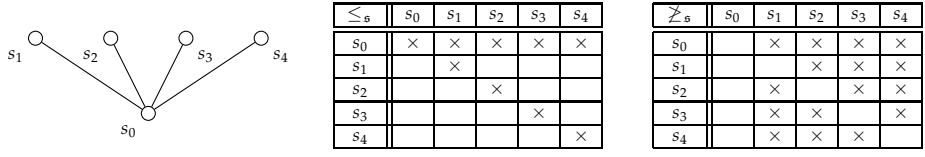
\begin{figure}
	\centering\scriptsize
	\begin{tikzpicture}
		\def\x{1};
		\def\y{1};
		\draw(2.5*\x,1*\y) node[draw,circle,scale=.8,label=below left:$s_{0}$](v0){};
		\draw(1*\x,2*\y) node[draw,circle,scale=.8,label=below left:$s_{1}$](v1){};
		\draw(2*\x,2*\y) node[draw,circle,scale=.8,label=below left:$s_{2}$](v2){};
		\draw(3*\x,2*\y) node[draw,circle,scale=.8,label=below right:$s_{3}$](v3){};
		\draw(4*\x,2*\y) node[draw,circle,scale=.8,label=below right:$s_{4}$](v4){};
		\draw(v0) -- (v1);
		\draw(v0) -- (v2);
		\draw(v0) -- (v3);
		\draw(v0) -- (v4);
		\draw(6.75*\x,1.5*\y) node{
		  \begin{tabular}{|c||c|c|c|c|c|}
			  \hline
			  $\leq_{\fs}$ & $s_{0}$ & $s_{1}$ & $s_{2}$ & $s_{3}$ & $s_{4}$\\
			  \hline\hline
			  $s_{0}$ & $\times$ & $\times$ & $\times$ & $\times$ & $\times$ \\
			  \hline
			  $s_{1}$ & & $\times$ & & & \\
			  \hline
			  $s_{2}$ & & & $\times$ & & \\
			  \hline
			  $s_{3}$ & & & & $\times$ & \\
			  \hline
			  $s_{4}$ & & & & & $\times$ \\
			  \hline
		  \end{tabular}
		};
		\draw(11*\x,1.5*\y) node{
		  \begin{tabular}{|c||c|c|c|c|c|}
			  \hline
			  $\not\geq_{\fs}$ & $s_{0}$ & $s_{1}$ & $s_{2}$ & $s_{3}$ & $s_{4}$\\
			  \hline\hline
			  $s_{0}$ & & $\times$ & $\times$ & $\times$ & $\times$ \\
			  \hline
			  $s_{1}$ & & & $\times$ & $\times$ & $\times$ \\
			  \hline
			  $s_{2}$ & & $\times$ & & $\times$ & $\times$ \\
			  \hline
			  $s_{3}$ & & $\times$ & $\times$ & & $\times$ \\
			  \hline
			  $s_{4}$ & & $\times$ & $\times$ & $\times$ & \\
			  \hline
		  \end{tabular}
		};
	\end{tikzpicture}
	\caption{A $4$-star, its incidence table and the corresponding contraordinal scale.}
	\label{fig:4_star}
\end{figure}

\subsection{$n$-Chains}
  \label{sec:n_chains}
Let $C=\{c_{1},c_{2},\ldots,c_{n}\}$ be a set. An \alert{$n$-chain} is a poset $\cc=(C,\leq_{\cc})$ satisfying
$c_{i}\leq_{\cc}c_{j}$ if and only if $i\leq j$ for all $i,j\in\{1,2,\ldots,n\}$. (See 
Figure~\ref{fig:4_chain} for an example.)
Clearly, the corresponding contraordinal scale $(C,C,\not\geq_{\cc})$ has precisely $n+1$
formal concepts, namely
\begin{displaymath}
	\bigl(\{c_{1},c_{2},\ldots,c_{i-1}\},\{c_{i},c_{i+1},\ldots,c_{n}\}\bigr)
	  \quad\mbox{for}\;i\in\{1,2,\ldots,n+1\}.
\end{displaymath}
(In the case $i=n+1$, the set $\{c_{i},c_{i+1},\ldots,c_{n}\}$ is to be interpreted as the empty set and in 
the case $i=1$, the set $\{c_{1},c_{2},\ldots,c_{i-1}\}$ is to be interpreted as the empty set.) See for 
instance \cite{muehle12counting}*{Section~3.1} for a more detailed explanation.

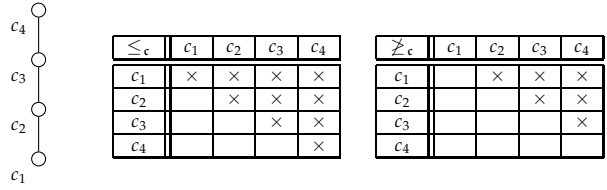
\begin{figure}
	\centering\scriptsize
	\begin{tikzpicture}
		\def\x{1};
		\def\y{.66};
		\draw(1*\x,1*\y) node[draw,circle,scale=.8,label=below left:$c_{1}$](v1){};
		\draw(1*\x,2*\y) node[draw,circle,scale=.8,label=below left:$c_{2}$](v2){};
		\draw(1*\x,3*\y) node[draw,circle,scale=.8,label=below left:$c_{3}$](v3){};
		\draw(1*\x,4*\y) node[draw,circle,scale=.8,label=below left:$c_{4}$](v4){};
		\draw(v1) -- (v2) -- (v3) -- (v4);
		\draw(3.5*\x,2.25*\y) node{
		  \begin{tabular}{|c||c|c|c|c|}
			  \hline
			  $\leq_{\cc}$ & $c_{1}$ & $c_{2}$ & $c_{3}$ & $c_{4}$\\
			  \hline\hline
			  $c_{1}$ & $\times$ & $\times$ & $\times$ & $\times$ \\
			  \hline
			  $c_{2}$ & & $\times$ & $\times$ & $\times$ \\
			  \hline
			  $c_{3}$ & & & $\times$ & $\times$ \\
			  \hline
			  $c_{4}$ & & & & $\times$ \\
			  \hline
		  \end{tabular}
		};
		\draw(7*\x,2.25*\y) node{
		  \begin{tabular}{|c||c|c|c|c|}
			  \hline
			  $\not\geq_{\cc}$ & $c_{1}$ & $c_{2}$ & $c_{3}$ & $c_{4}$\\
			  \hline\hline
			  $c_{1}$ & & $\times$ & $\times$ & $\times$ \\
			  \hline
			  $c_{2}$ & & & $\times$ & $\times$ \\
			  \hline
			  $c_{3}$ & & & & $\times$ \\
			  \hline
			  $c_{4}$ & & & & \\
			  \hline
		  \end{tabular}
		};
	\end{tikzpicture}
	\caption{A $4$-chain, its incidence table and the corresponding contraordinal scale.}
	\label{fig:4_chain}
\end{figure}

\subsection{Convolutions}
  \label{sec:convolution}
Let $u=(u_{1},u_{2},\ldots,u_{k})$ and $v=(v_{1},v_{2},\ldots,v_{k})$ be two vectors of length $k$. The 
\alert{convolution of $u$ and $v$} is defined as 
\begin{align*}
	u\star v=\sum_{i=1}^{k}{u_{i}\cdot v_{k-i+1}}.
\end{align*}

\begin{figure}
	\centering\scriptsize
	\begin{tabular}{c||c|c|c|c|c|c}
		& $j=1$ & $j=2$ & $j=3$ & $j=4$ & $j=5$ & $j=6$\\
		\hline\hline
		$i=1$ & $1$ & $8$ & $34$ & $104$ & $259$ & $560$ \\
		$i=2$ & $4$ & $25$ & $88$ & $234$ & $524$ & $1043$ \\
		$i=3$ & $9$ & $52$ & $170$ & $424$ & $899$ & $1708$ \\
		$i=4$ & $16$ & $89$ & $280$ & $674$ & $1384$ & $2555$ \\
	\end{tabular}
	\caption{The first four rows and six columns of the convolution array of $u_{2}$ and $v_{2}$.}
	\label{fig:convolution_array}
\end{figure}

In this article, we are interested in the convolutions of two very special vectors, given by functions
$u_{m}(h)=h^m$ and $v_{m}(i,h)=(i-1+h)^{m}$. Define the convolution array of $u_{m}$ and $v_{m}$ as the 
rectangular array whose entries $a_{i,j}$ are defined as 
\begin{align*}
	a_{i,j} & =\Bigl(u_{m}(1),u_{m}(2),\ldots,u_{m}(j)\Bigr)
	  \star\Bigl(v_{m}(i,1),v_{m}(i,2),\ldots,v_{m}(i,j)\Bigr)\\
	& = \sum_{k=1}^{j}{u_{m}(k)\cdot v_{m}(i,j-k+1)}\\
	& = \sum_{k=1}^{j}{\bigl(k(i+j-k)\bigr)^{m}}
\end{align*}
See Figure~\ref{fig:convolution_array} for an illustration. In the cases $m=2$ and $m=3$ we recover 
\cite{sloane}*{A213505} and \cite{sloane}*{A213558} respectively.
However, we are not interested in the whole convolution array, but in the sums of the antidiagonals. Define 
\begin{align}\label{eq:count}
	C(m,n) & =\sum_{l=1}^{n}{a_{l,n-l+1}}\\
	\nonumber & = \sum_{l=1}^{n}{\sum_{k=1}^{n-l+1}{\bigl(k(n-k+1)\bigr)^{m}}}\\
	\nonumber & = \sum_{k=1}^{n}{k^{m}(n-k+1)^{m+1}}
\end{align}
to be the sum of the $n$-th antidiagonal of the convolution array of $u_{m}$ and $v_{m}$. The first few 
entries of the sequence $C(2,n)$ (starting with $n=0$) are 
\begin{displaymath}
	0,1,12,68,260,777,1960,4368,\ldots,
\end{displaymath}
see \cite{sloane}*{A213547}, and the first few entries of the sequence $C(3,n)$ (starting with 
$n=0$) are
\begin{displaymath}
	0,1,24,236,1400,6009,20608,59952,\ldots,
\end{displaymath}
see \cite{sloane}*{A213560}. In view of \eqref{eq:count}, proving Theorem~\ref{thm:cardinality} is
equivalent to showing that 
\begin{equation}
  \label{eq:enumeration}
	\bigl\lvert\mathfrak{S}\!\mathfrak{C}^{\bullet}_{m,n}\bigr\rvert=C(m,n+1).
\end{equation}

\section{Embedding $\mathfrak{A\!C}_{m,n}^{\bullet}$ into $\mathfrak{S\!C}_{m,n}^{\bullet}$}
  \label{sec:embedding}
In order to prove Theorem~\ref{thm:cardinality}, we make use of the following observation. Let
$\mathfrak{A\!C}_{m,n}^{\bullet}$ denote the set of proper mergings of an $m$-antichain and an $n$-chain.

\begin{proposition}
  \label{prop:quotient_lattice}
	The lattice $\bigl(\mathfrak{A\!C}_{m,n}^{\bullet},\preceq\bigr)$ is a quotient lattice of 
	$\bigl(\mathfrak{S\!C}_{m,n}^{\bullet},\preceq\bigr)$.
\end{proposition}

Let $\ac=(A,=_{\ac}),\fs=(S,\leq_{\fs}),$ and $\cc=(C,\leq_{\cc})$ be an $m$-antichain, an $m$-star and an 
$n$-chain, respectively, as defined in Sections~\ref{sec:m_stars} and \ref{sec:n_chains}. If we consider the 
restriction $(S\setminus\{s_{0}\},\leq_{\fs})$ we implicitly understand the partial order $\leq_{\fs}$ to be 
restricted to the ground set $A=S\setminus\{s_{0}\}$. Hence, we identify the posets 
$(S\setminus\{s_{0}\},\leq_{\fs})$ and $(A,=_{\ac})$. If we write $S=\{s_{0},s_{1},\ldots,s_{m}\}$, then we 
identify $s_{i}=a_{i}$ for $i\in\{1,2,\ldots,m\}$.

Now let $(R,T)\in\mathfrak{S\!C}_{m,n}^{\bullet}$ and consider the restrictions 
$\overline{R}=R\cap(A\times C)$, and $\overline{T}=T\cap(C\times A)$. Further, if 
$(R,T)\in\mathfrak{A\!C}_{m,n}^{\bullet}$, then define a pair of relations $(R_{o},T_{o})$ with 
$R_{o}\subseteq S\times C$ and $T_{o}\subseteq C\times S$ in the following way:
\begin{align*}
	s\;R_{o}\;c_{j}\quad\mbox{if and only if}\quad & 
		\begin{cases}
			s=s_{0} & \mbox{and there exists some}\;i\in\{1,2,\ldots,m\}\\
			  & \mbox{with}\;a_{i}\;R\;c_{j},\\
			s=a_{i} & \mbox{for some}\;i\in\{1,2,\ldots,m\}\;\mbox{and}\;
			  a_{i}\;R\;c_{j},
		\end{cases} \\
	c_{j}\;T_{o}\;s\quad\mbox{if and only if}\quad &
	  s=a_{i}\;\mbox{for some}\;i\in\{1,2,\ldots,m\}\;\mbox{and}\;c_{j}\;T\;a_{i}.
\end{align*}

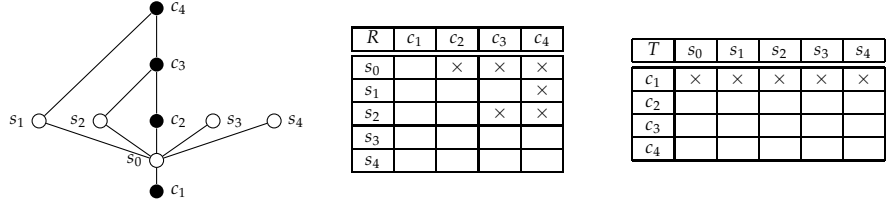
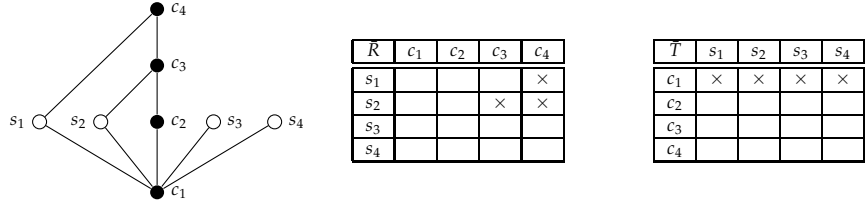
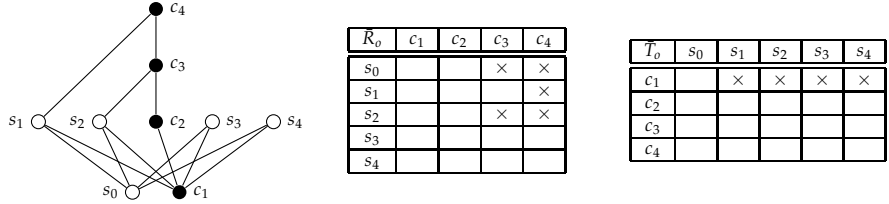
\begin{figure}[t]
	\subfigure[A proper merging of a $4$-star and a $4$-chain, and the corresponding relations 
	  $R$ and $T$.]{\label{fig:restriction_injection_1}
		\begin{tikzpicture}\scriptsize
			\draw(0,0) node{
				\begin{tikzpicture}
					\def\x{1.25};
					\def\y{.75};
					\draw(2.5*\x,1.75*\y) 
					  node[fill=black,circle,scale=.8,label=right:$c_{1}$](c1){};
					\draw(2.5*\x,3*\y) 
					  node[fill=black,circle,scale=.8,label=right:$c_{2}$](c2){};
					\draw(2.5*\x,4*\y) 
					  node[fill=black,circle,scale=.8,label=right:$c_{3}$](c3){};
					\draw(2.5*\x,5*\y) 
					  node[fill=black,circle,scale=.8,label=right:$c_{4}$](c4){};
					\draw(2.5*\x,2.3*\y) 
					  node[draw,circle,scale=.8,label=left:$s_{0}$](s0){};
					\draw(1.25*\x,3*\y) 
					  node[draw,circle,scale=.8,label=left:$s_{1}$](s1){};
					\draw(1.9*\x,3*\y) 
					  node[draw,circle,scale=.8,label=left:$s_{2}$](s2){};
					\draw(3.1*\x,3*\y) 
					  node[draw,circle,scale=.8,label=right:$s_{3}$](s3){};
					\draw(3.75*\x,3*\y) 
					  node[draw,circle,scale=.8,label=right:$s_{4}$](s4){};
					\draw(c1) -- (s0) -- (c2) -- (c3) -- (c4);
					\draw(s0) -- (s1);
					\draw(s0) -- (s2);
					\draw(s0) -- (s3);
					\draw(s0) -- (s4);
					\draw(s1) -- (c4);
					\draw(s2) -- (c3);
				\end{tikzpicture}
			};
			\draw(4,0) node{
				\begin{tabular}{|c|c|c|c|c|}
					\hline
					$R$ & $c_{1}$ & $c_{2}$ & $c_{3}$ & $c_{4}$ \\
					\hline\hline
					$s_{0}$ & & $\times$ & $\times$ & $\times$ \\
					\hline
					$s_{1}$ & & & & $\times$ \\
					\hline
					$s_{2}$ & & & $\times$ & $\times$ \\
					\hline
					$s_{3}$ & & & & \\
					\hline
					$s_{4}$ & & & & \\
					\hline
				\end{tabular}
			};
			\draw(8,0) node{
				\begin{tabular}{|c|c|c|c|c|c|}
					\hline
					$T$ & $s_{0}$ & $s_{1}$ & $s_{2}$ & $s_{3}$ & $s_{4}$ \\
					\hline\hline
					$c_{1}$ & $\times$ & $\times$ & $\times$ & $\times$ & $\times$ \\
					\hline
					$c_{2}$ & & & & & \\
					\hline
					$c_{3}$ & & & & & \\
					\hline
					$c_{4}$ & & & & & \\
					\hline
				\end{tabular}
			};
		\end{tikzpicture}
	}
	\subfigure[The image of $(R,T)$ from Figure~\ref{fig:restriction_injection_1} under the map $\eta$ 
	  is a proper merging of a $4$-antichain and a $4$-chain.]{\label{fig:restriction_injection_2}
		\begin{tikzpicture}\scriptsize
			\draw(0,0) node{
				\begin{tikzpicture}
					\def\x{1.25};
					\def\y{.75};
					\draw(2.5*\x,1.75*\y) 
					  node[fill=black,circle,scale=.8,label=right:$c_{1}$](c1){};
					\draw(2.5*\x,3*\y) 
					  node[fill=black,circle,scale=.8,label=right:$c_{2}$](c2){};
					\draw(2.5*\x,4*\y) 
					  node[fill=black,circle,scale=.8,label=right:$c_{3}$](c3){};
					\draw(2.5*\x,5*\y) 
					  node[fill=black,circle,scale=.8,label=right:$c_{4}$](c4){};
					\draw(1.25*\x,3*\y) 
					  node[draw,circle,scale=.8,label=left:$s_{1}$](s1){};
					\draw(1.9*\x,3*\y) 
					  node[draw,circle,scale=.8,label=left:$s_{2}$](s2){};
					\draw(3.1*\x,3*\y) 
					  node[draw,circle,scale=.8,label=right:$s_{3}$](s3){};
					\draw(3.75*\x,3*\y) 
					  node[draw,circle,scale=.8,label=right:$s_{4}$](s4){};
					\draw(c1) -- (c2) -- (c3) -- (c4);
					\draw(c1) -- (s1);
					\draw(c1) -- (s2);
					\draw(c1) -- (s3);
					\draw(c1) -- (s4);
					\draw(s1) -- (c4);
					\draw(s2) -- (c3);
				\end{tikzpicture}
			};
			\draw(4,0) node{
				\begin{tabular}{|c|c|c|c|c|}
					\hline
					$\bar{R}$ & $c_{1}$ & $c_{2}$ & $c_{3}$ & $c_{4}$ \\
					\hline\hline
					$s_{1}$ & & & & $\times$ \\
					\hline
					$s_{2}$ & & & $\times$ & $\times$ \\
					\hline
					$s_{3}$ & & & & \\
					\hline
					$s_{4}$ & & & & \\
					\hline
				\end{tabular}
			};
			\draw(8,0) node{
				\begin{tabular}{|c|c|c|c|c|}
					\hline
					$\bar{T}$ & $s_{1}$ & $s_{2}$ & $s_{3}$ & $s_{4}$ \\
					\hline\hline
					$c_{1}$ & $\times$ & $\times$ & $\times$ & $\times$ \\
					\hline
					$c_{2}$ & & & & \\
					\hline
					$c_{3}$ & & & & \\
					\hline
					$c_{4}$ & & & & \\
					\hline
				\end{tabular}
			};
		\end{tikzpicture}
	}
	\subfigure[The image of $(\bar{R},\bar{T})$ from  Figure~\ref{fig:restriction_injection_2} under the 
	  injection $\xi$ is again a proper merging of a $4$-star and a $4$-chain.]{
	  \label{fig:restriction_injection_3}
		\begin{tikzpicture}\scriptsize
			\draw(0,0) node{
				\begin{tikzpicture}
					\def\x{1.25};
					\def\y{.75};
					\draw(2.75*\x,1.75*\y) 
					  node[fill=black,circle,scale=.8,label=right:$c_{1}$](c1){};
					\draw(2.5*\x,3*\y) 
					  node[fill=black,circle,scale=.8,label=right:$c_{2}$](c2){};
					\draw(2.5*\x,4*\y) 
					  node[fill=black,circle,scale=.8,label=right:$c_{3}$](c3){};
					\draw(2.5*\x,5*\y) 
					  node[fill=black,circle,scale=.8,label=right:$c_{4}$](c4){};
					\draw(2.25*\x,1.75*\y) 
					  node[draw,circle,scale=.8,label=left:$s_{0}$](s0){};
					\draw(1.25*\x,3*\y) 
					  node[draw,circle,scale=.8,label=left:$s_{1}$](s1){};
					\draw(1.9*\x,3*\y) 
					  node[draw,circle,scale=.8,label=left:$s_{2}$](s2){};
					\draw(3.1*\x,3*\y) 
					  node[draw,circle,scale=.8,label=right:$s_{3}$](s3){};
					\draw(3.75*\x,3*\y) 
					  node[draw,circle,scale=.8,label=right:$s_{4}$](s4){};
					\draw(c1) -- (c2) -- (c3) -- (c4);
					\draw(s0) -- (s1);
					\draw(s0) -- (s2);
					\draw(s0) -- (s3);
					\draw(s0) -- (s4);
					\draw(c1) -- (s1);
					\draw(c1) -- (s2);
					\draw(c1) -- (s3);
					\draw(c1) -- (s4);
					\draw(s1) -- (c4);
					\draw(s2) -- (c3);
				\end{tikzpicture}
			};
			\draw(4,0) node{
				\begin{tabular}{|c|c|c|c|c|}
					\hline
					$\bar{R}_{o}$ & $c_{1}$ & $c_{2}$ & $c_{3}$ & $c_{4}$ \\
					\hline\hline
					$s_{0}$ & & & $\times$ & $\times$ \\
					\hline
					$s_{1}$ & & & & $\times$ \\
					\hline
					$s_{2}$ & & & $\times$ & $\times$ \\
					\hline
					$s_{3}$ & & & & \\
					\hline
					$s_{4}$ & & & & \\
					\hline
				\end{tabular}
			};
			\draw(8,0) node{
				\begin{tabular}{|c|c|c|c|c|c|}
					\hline
					$\bar{T}_{o}$ & $s_{0}$ & $s_{1}$ & $s_{2}$ & $s_{3}$ & $s_{4}$ \\
					\hline\hline
					$c_{1}$ & & $\times$ & $\times$ & $\times$ & $\times$ \\
					\hline
					$c_{2}$ & & & & & \\
					\hline
					$c_{3}$ & & & & & \\
					\hline
					$c_{4}$ & & & & & \\
					\hline
				\end{tabular}
			};
		\end{tikzpicture}
	}
	\caption{An illustration of the maps $\xi$ and $\eta$.}
	\label{fig:restriction_injection}
\end{figure}

We notice that $T$ and $T_{o}$ coincide as sets, but they differ as cross-tables, since $T_{o}$ has an 
additional (but empty) column. $R_{o}$ can be viewed as a copy of the cross-table of $R$, where the union of 
the rows of $R$ is added again as first row. Now let us define two maps
\begin{align}
  \label{eq:eta}
	\eta & : \mathfrak{S\!C}_{m,n}^{\bullet}\to\mathfrak{A\!C}_{m,n}^{\bullet},\quad 
	  (R,T)\mapsto(\overline{R},\overline{T}),\quad\mbox{and}\\
  \label{eq:xi}
	\xi & : \mathfrak{A\!C}_{m,n}^{\bullet}\to\mathfrak{S\!C}_{m,n}^{\bullet},\quad 
	  (R,T)\mapsto(R_{o},T_{o}).
\end{align}
See Figure~\ref{fig:restriction_injection} for an illustration. We have to show that $\eta$ and $\xi$ 
are well-defined.

\begin{lemma}
  \label{lem:restriction}
	If $(R,T)\in\mathfrak{S\!C}_{m,n}^{\bullet}$, then 
	$(\overline{R},\overline{T})\in\mathfrak{A\!C}_{m,n}^{\bullet}$. 
\end{lemma}
\begin{proof}
	Write $A=S\setminus\{s_{0}\}$, and let $(R,T)\in\mathfrak{S\!C}_{m,n}^{\bullet}$. We need to 
	show that $(\overline{R},\overline{T})$ satisfies the conditions from 
	Proposition~\ref{prop:classification_mergings}. First of all, we want to show that $\overline{R}$ is 
	a bond from $(A,A,\neq_{\ac})$ to $(C,C,\not\geq_{\cc})$, and we know that 
	$R\subseteq S\times C$ is a bond from $(S,S,\not\geq_{\fs})$ to $(C,C,\not\geq_{\cc})$. By 
	construction, $\overline{R}\subseteq A\times C$, and we have 
	$a_{i}^{\overline{R}}=a_{i}^{R}$ for $i\in\{1,2,\ldots,m\}$, thus every row of $\overline{R}$ is an 
	intent of $(C,C,\not\geq_{\cc})$. Now let $c\in C$. By definition, we know that $c^{R}$ is an 
	extent of $(S,S,\not\geq_{\fs})$. It follows from the reasoning in Section~\ref{sec:m_stars} that
	either $c^{R}=\emptyset$ or $c^{R}=B\cup\{s_{0}\}$ for some $B\subseteq A$.
	Hence, $c^{\overline{R}}=\emptyset$ or $c^{\overline{R}}=B$ for some $B\subseteq A$. 
	Since $(A,\neq_{\ac})$ is an antichain, the contraordinal scale $(A,A,\neq_{\ac})$ is known to be 
	isomorphic to the formal context of the Boolean lattice with $2^{m}$ elements, and 
	$c^{\overline{R}}$ is thus an extent of this context. The fact that $\overline{T}$ is a bond from 
	$(C,C,\not\geq_{\cc})$ to $(A,A,\neq_{\ac})$ follows analogously.
	
	It is easy to see that $\bigl(\overline{R}\circ\overline{T}\bigr)\subseteq (R\circ T)$ and 
	$\bigl(\overline{T}\circ\overline{R}\bigr)\subseteq (T\circ R)$, proving the remaining two conditions.
\end{proof}

\begin{lemma}
  \label{lem:injection}
	If $(R,T)\in\mathfrak{A\!C}_{m,n}^{\bullet}$, then $(R_{o},T_{o})\in\mathfrak{S\!C}_{m,n}^{\bullet}$.
\end{lemma}
\begin{proof}
	Let $S=A\cup\{s_{0}\}$, where $A=\{a_{1},a_{2},\ldots,a_{m}\}$ is the ground set of the antichain 
	$\ac=(A,=_{\ac})$. For every $i\in\{1,2,\ldots,m\}$, we have $a_{i}^{R_{o}}=a_{i}^{R}$. Since 
	$R$ is a bond from $(A,A,\neq_{\ac})$ to $(C,C,\not\geq_{\cc})$, we find that $a_{i}\;R\;c_{j}$
	implies $a_{i}\;R\;c_{k}$ for all $k\geq j$. Hence, $s_{0}^{R_{o}}=a_{i}^{R}$ for some 
	$a_{i}\in A$, and thus every row of $R_{o}$ is an intent of $(C,C,\not\geq_{\cc})$. If
	$c\in C$, then by construction $c^{R_{o}}=\emptyset$ or $c^{R_{o}}=c^{R}\cup\{s_{0}\}$, and thus 
	every column of $R_{o}$ is an extent of $(S,S,\not\geq_{\fs})$. For every $i\in\{1,2,\ldots,m\}$, we 
	have $a_{i}^{T_{o}}=a_{i}^{T}$, and $s_{0}^{T_{o}}=\emptyset$. Hence, every column of $T_{o}$ 
	is an extent of $(C,C,\not\geq_{\cc})$. Moreover, for $c\in C$, we have $c^{T_{o}}=c^{T}$, 
	and thus every row of $T_{o}$ is an intent of $(S,S,\not\geq_{\fs})$.
	
	Consider the relational product $R_{o}\circ T_{o}$, and let $(s,s')\in R_{o}\circ T_{o}$. By 
	definition, there exists some $c\in C$ with $s\;R_{o}\;c$ and $c\;T_{o}\;s'$. By construction, 
	$s'\neq s_{0}$, and for every pair $(s,s')\in R_{o}\circ T_{o}$ with $s\neq s_{0}$, we have 
	$(s,s')\in R\circ T$, and thus $s=s'$, since $R\circ T$ is contained in $=_{\ac}$. This is, however,
	a contradiction to $R\cap T^{-1}=\emptyset$. Thus, $R_{o}\circ T_{o}$ can only contain pairs of the 
	form $(s_{0},s')$. These pairs satisfy $s_{0}\leq_{\fs}s'$ by definition of the order 
	relation $\leq_{\fs}$, and we conclude that $R_{o}\circ T_{o}$ is contained in 
	$\leq_{\fs}$. Now let $(c,c')\in T_{o}\circ R_{o}$, and let $s\in S$ with $c\;T_{o}\;s$ and 
	$s\;R_{o}\;c'$. By construction, $T_{o}$ does not contain a pair of the form
	$(c,s_{0})$, and if $s\neq s_{0}$, then $c\leq_{\cc} c'$ since $T\circ R$ is 
	contained in $\leq_{\cc}$, which completes the proof.
\end{proof}

Let us collect some properties of $\eta$ and $\xi$. 

\begin{lemma}
  \label{lem:properties_eta_xi}
	The map $\eta$ is surjective, and the map $\xi$ is injective.
\end{lemma}
\begin{proof}
	Let $(R,T)\in\mathfrak{A\!C}_{m,n}^{\bullet}$, and let $(R_{o},T_{o})=\xi(R,T)$. By construction, 
	$R_{o}$ arises from $R$ by adding elements of the form $(s_{0},\cdot)$, and $T_{o}=T$. Consider
	$(\overline{R}_{o},\overline{T}_{o})=\eta(R_{o},T_{o})$. By construction, $\overline{R}_{o}$ contains
	all elements in $R_{o}$, except those of the form $(s_{0},\cdot)$, and analogously for
	$\overline{T}_{o}$. Thus, $(\overline{R}_{o},\overline{T}_{o})=(R,T)$, and we conclude that 
	$\eta\circ\xi=\mbox{Id}_{\mathfrak{A\!C}_{m,n}^{\bullet}}$.
	
	Suppose there exists $(R,T)\in\mathfrak{A\!C}_{m,n}^{\bullet}$ with $(R,T)\notin\mbox{Im}(\eta)$.
	By definition, we have $\xi(R,T)\in\mathfrak{S\!C}_{m,n}^{\bullet}$, and thus 
	$\eta(\xi(R,T))\in\mathfrak{A\!C}_{m,n}^{\bullet}$. We have shown in the previous paragraph that
	$\eta(\xi(R,T))=(R,T)$, which contradicts $(R,T)\notin\mbox{Im}(\eta)$. Thus, $\eta$ is surjective.
	
	Now let $(R_{1},T_{1}),(R_{2},T_{2})\in\mathfrak{A\!C}_{m,n}^{\bullet}$ with 
	$\xi(R_{1},T_{1})=\xi(R_{2},T_{2})$. Since $\eta$ is a map, this implies that
	$\eta(\xi(R_{1},T_{1}))=\eta(\xi(R_{2},T_{2}))$, and we obtain with the reasoning in the first 
	paragraph that $(R_{1},T_{1})=(R_{2},T_{2})$. Thus, $\xi$ is injective.
\end{proof}

\begin{proposition}
  \label{prop:xi_homomorphism}
	The maps $\eta$ and $\xi$ defined in \eqref{eq:eta} and \eqref{eq:xi} are order-preserving 
	lattice-homomorphisms. 
\end{proposition}
\begin{proof}
	Let us start with $\eta$, and let $(R_{1},T_{1}),(R_{2},T_{2})\in\mathfrak{S\!C}_{m,n}^{\bullet}$ be
	two proper mergings of an $m$-star and an $n$-chain, satisfying $(R_{1},T_{1})\preceq(R_{2},T_{2})$.
	This means by definition of $\preceq$, see \eqref{eq:merging_order}, that 
	$R_{1}\subseteq R_{2}$ and $T_{1}\supseteq T_{2}$. By definition of $\eta$, we have 
	$\overline{R}_{i}=R_{i}\setminus\{s_{0}^{R_{i}}\}$ and
	$\overline{T}_{i}=T_{i}\setminus\{s_{0}^{T_{i}}\}$ for $i\in\{1,2\}$. Thus, it follows immediately that 
	$(\overline{R}_{1},\overline{T}_{1})\preceq(\overline{R}_{2},\overline{T}_{2})$.
	
	For showing that $\eta$ is a lattice-homomorphism, we need to show that it is compatible with the 
	lattice operations. This means, we need to show that for every 
	$(R_{1},T_{1}),(R_{2},T_{2})\in\mathfrak{S\!C}_{m,n}^{\bullet}$, we have
	\begin{align*}
		\eta\bigl((R_{1},T_{1})\vee(R_{2},T_{2})\bigr)
		  & =\eta\bigl((R_{1},T_{1})\bigr)\vee\eta\bigl((R_{2},T_{2})\bigr),\quad\mbox{and}\\
		\eta\bigl((R_{1},T_{1})\wedge(R_{2},T_{2})\bigr)
		  & =\eta\bigl((R_{1},T_{1})\bigr)\wedge\eta\bigl((R_{2},T_{2})\bigr).
	\end{align*}
	It was shown in \cite{ganter11merging}*{Theorem~1} that
	\begin{align*}
		(R_{1},T_{1})\vee(R_{2},T_{2}) & =(R_{1}\cup R_{2},T_{1}\cap T_{2}),\quad\mbox{and}\\
		(R_{1},T_{1})\wedge(R_{2},T_{2}) & = (R_{1}\cap R_{2},T_{1}\cup T_{2}).
	\end{align*}
	Thus, we have to show that 
	\begin{align*}
		\bigl(\overline{R_{1}\cup R_{2}},\overline{T_{1}\cap T_{2}}\bigr)
		  & = \bigl(\overline{R}_{1}\cup\overline{R}_{2},
		  \overline{T}_{1}\cap\overline{T}_{2}\bigr),\quad\mbox{and}\\
		\bigl(\overline{R_{1}\cap R_{2}},\overline{T_{1}\cup T_{2}}\bigr)
		  & = \bigl(\overline{R}_{1}\cap\overline{R}_{2},
		  \overline{T}_{1}\cup\overline{T}_{2}\bigr).
	\end{align*}
	Since $\overline{(\cdot)}$ is a restriction operator, these equalities are trivially satisfied.
	
	\smallskip
	
	Let now $(R_{1},T_{1}),(R_{2},T_{2})\in\mathfrak{A\!C}_{m,n}^{\bullet}$ be two proper mergings of an
	$m$-antichain and an $n$-chain, satisfying $(R_{1},T_{1})\preceq(R_{2},T_{2})$. 
	By construction, $(T_{i})_{o}=T_{i}$ (considered as sets) for $i\in\{1,2\}$. Moreover, for 
	$i\in\{1,2\}$, the set $(R_{i})_{o}$ is obtained from $R_{i}$ by adding pairs $(s_{0},c_{k})$
	for all $c_{k}\in C$ satisfying $a_{j}\;R_{i}\;c_{k}$ for some $a_{j}\in A$. 
	If $R_{1}\subseteq R_{2}$, then it is clear that $(R_{2})_{o}$ has at least as many additional 
	relations as $(R_{1})_{o}$, hence implying $(R_{1})_{o}\subseteq (R_{2})_{o}$. This proves
	$\bigl((R_{1})_{o},(T_{1})_{o}\bigr)\preceq\bigl((R_{2})_{o},(T_{2})_{o}\bigr)$, which implies that
	$\xi$ is order-preserving.
	
	With the reasoning from above, showing that $\xi$ is a lattice-homomorphism reduces to showing 
	that for every $(R_{1},T_{1}),(R_{2},T_{2})\in\mathfrak{A\!C}_{m,n}^{\bullet}$, we have
	\begin{align*}
		\bigl((R_{1}\cup R_{2})_{o},(T_{1}\cap T_{2})_{o}\bigr)
		  & = \bigl((R_{1})_{o}\cup (R_{2})_{o},(T_{1})_{o}\cap (T_{2})_{o}\bigr),\quad\mbox{and}\\
		\bigl((R_{1}\cap R_{2})_{o},(T_{1}\cup T_{2})_{o}\bigr)
		  & = \bigl((R_{1})_{o}\cap (R_{2})_{o},(T_{1})_{o}\cup (T_{2})_{o}\bigr).
	\end{align*}
	Since by construction $(T_{i})_{o}=T_{i}$ for $i\in\{1,2\}$, we can restrict our attention to 
	the relations $R_{1}$ and $R_{2}$, and it is sufficient to focus on the behavior of 
	$s_{0}^{R_{1}}$ and $s_{0}^{R_{2}}$, since the other rows remain unchanged. Clearly,
	$(s_{0},c)\in(R_{1}\cup R_{2})_{o}$ is equivalent to the existence of some $a\in A$ with
	$a\;R_{1}\;c$ or $a\;R_{2}\;c$, which means that $(s_{0},c)\in (R_{1})_{o}\cup (R_{2})_{o}$.
	Similarly, $(s_{0},c)\in(R_{1}\cap R_{2})_{o}$ is equivalent to the existence of some $a\in A$ with
	$a\;R_{1}\;c$ and $a\;R_{2}\;c$, which means that $(s_{0},c)\in (R_{1})_{o}\cap (R_{2})_{o}$, and 
	we are done.
\end{proof}

\begin{proof}[Proof of Proposition~\ref{prop:quotient_lattice}]
	Lemma~\ref{lem:properties_eta_xi} and Proposition~\ref{prop:xi_homomorphism} imply that $\eta$ is a 
	surjective lattice homomorphism from $\bigl(\mathfrak{S\!C}_{m,n}^{\bullet},\preceq\bigr)$ to
	$\bigl(\mathfrak{A\!C}_{m,n}^{\bullet},\preceq\bigr)$. Then, the Homomorphism Theorem for lattices,
	see for instance \cite{davey02introduction}*{Theorem~6.9}, implies the result.
\end{proof}

A consequence of Proposition~\ref{prop:quotient_lattice} is that for 
$(R,T)\in\mathfrak{A\!C}_{m,n}^{\bullet}$ the fiber $\eta^{-1}(R,T)$ is an interval in 
$\bigl(\mathfrak{S\!C}_{m,n}^{\bullet},\preceq\bigr)$, and all the fibers of $\eta$ are 
disjoint. We will use this property for the enumeration of the proper mergings of an $m$-star and an 
$n$-chain in the next section. Figure~\ref{fig:mergings_chain_star} shows the lattice of proper mergings 
of a $3$-star and a $1$-chain, and the shaded edges indicate how the lattice of proper mergings of a 
$3$-antichain and a $1$-chain arises as a quotient lattice. 
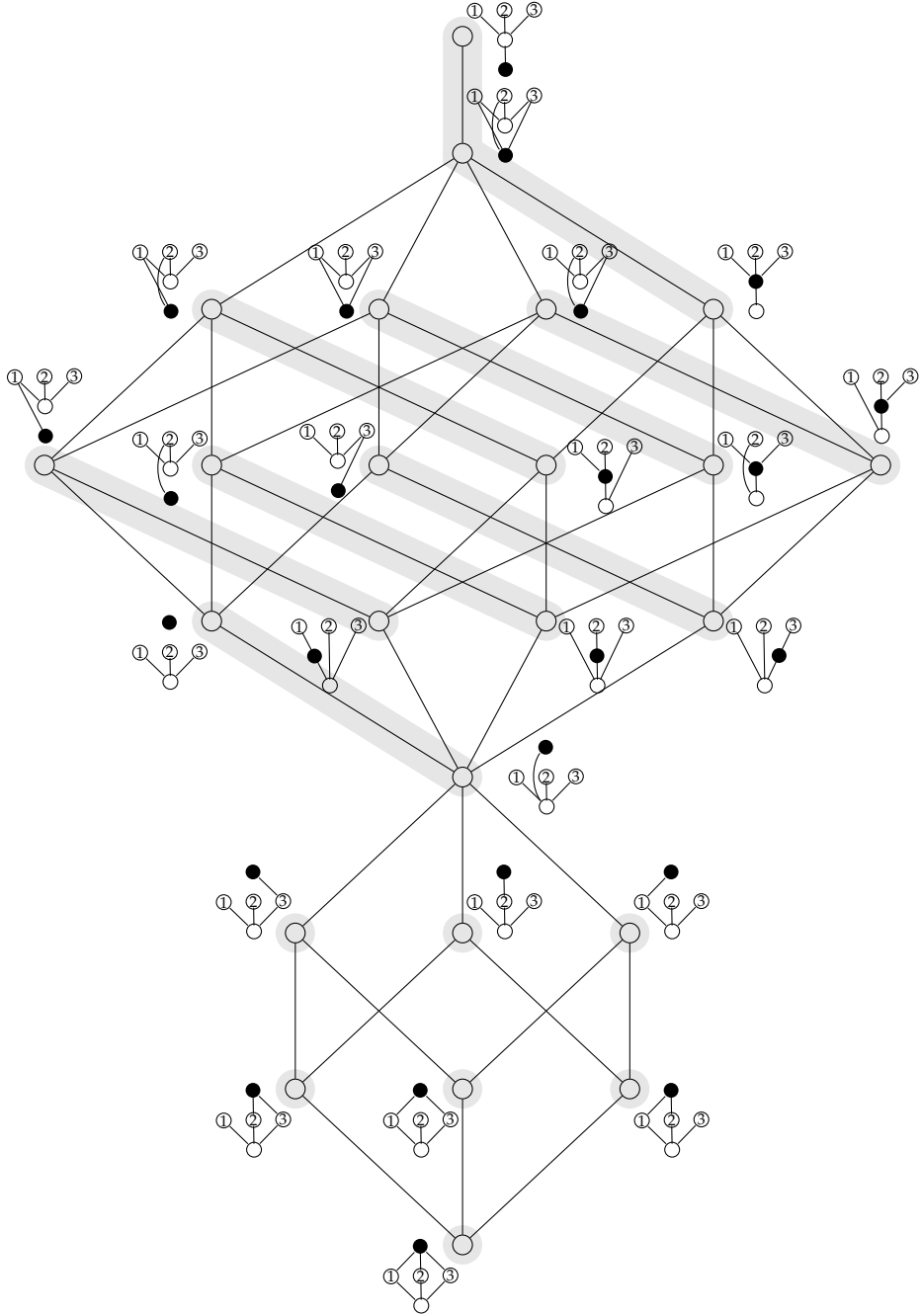
\begin{figure}
	\centering
	\begin{tikzpicture}
	\def\x{2.25};
	\def\y{2.1};
	\def\r{.8};
	\def\rs{.8};
	\draw(3.5*\x,1*\y) node[circle,draw,scale=\r](v1){};
	\draw(2.5*\x,2*\y) node[circle,draw,scale=\r](v2){};
	\draw(3.5*\x,2*\y) node[circle,draw,scale=\r](v3){};
	\draw(4.5*\x,2*\y) node[circle,draw,scale=\r](v4){};
	\draw(2.5*\x,3*\y) node[circle,draw,scale=\r](v5){};
	\draw(3.5*\x,3*\y) node[circle,draw,scale=\r](v6){};
	\draw(4.5*\x,3*\y) node[circle,draw,scale=\r](v7){};
	\draw(3.5*\x,4*\y) node[circle,draw,scale=\r](v8){};
	\draw(2*\x,5*\y) node[circle,draw,scale=\r](v9){};
	\draw(3*\x,5*\y) node[circle,draw,scale=\r](v10){};
	\draw(4*\x,5*\y) node[circle,draw,scale=\r](v11){};
	\draw(5*\x,5*\y) node[circle,draw,scale=\r](v12){};
	\draw(1*\x,6*\y) node[circle,draw,scale=\r](v13){};
	\draw(2*\x,6*\y) node[circle,draw,scale=\r](v14){};
	\draw(3*\x,6*\y) node[circle,draw,scale=\r](v15){};
	\draw(4*\x,6*\y) node[circle,draw,scale=\r](v16){};
	\draw(5*\x,6*\y) node[circle,draw,scale=\r](v17){};
	\draw(6*\x,6*\y) node[circle,draw,scale=\r](v18){};
	\draw(2*\x,7*\y) node[circle,draw,scale=\r](v19){};
	\draw(3*\x,7*\y) node[circle,draw,scale=\r](v20){};
	\draw(4*\x,7*\y) node[circle,draw,scale=\r](v21){};
	\draw(5*\x,7*\y) node[circle,draw,scale=\r](v22){};
	\draw(3.5*\x,8*\y) node[circle,draw,scale=\r](v23){};
	\draw(3.5*\x,8.75*\y) node[circle,draw,scale=\r](v24){};
	\draw(v1) -- (v2);
	\draw(v1) -- (v3);
	\draw(v1) -- (v4);
	\draw(v2) -- (v5);
	\draw(v2) -- (v6);
	\draw(v3) -- (v5);
	\draw(v3) -- (v7);
	\draw(v4) -- (v6);
	\draw(v4) -- (v7);
	\draw(v5) -- (v8);
	\draw(v6) -- (v8);
	\draw(v7) -- (v8);
	\draw(v8) -- (v9);
	\draw(v8) -- (v10);
	\draw(v8) -- (v11);
	\draw(v8) -- (v12);
	\draw(v9) -- (v13);
	\draw(v9) -- (v14);
	\draw(v9) -- (v15);
	\draw(v10) -- (v13);
	\draw(v10) -- (v16);
	\draw(v10) -- (v17);
	\draw(v11) -- (v14);
	\draw(v11) -- (v16);
	\draw(v11) -- (v18);
	\draw(v12) -- (v15);
	\draw(v12) -- (v17);
	\draw(v12) -- (v18);
	\draw(v13) -- (v19);
	\draw(v13) -- (v20);
	\draw(v14) -- (v19);
	\draw(v14) -- (v21);
	\draw(v15) -- (v20);
	\draw(v15) -- (v21);
	\draw(v16) -- (v19);
	\draw(v16) -- (v22);
	\draw(v17) -- (v20);
	\draw(v17) -- (v22);
	\draw(v18) -- (v21);
	\draw(v18) -- (v22);
	\draw(v19) -- (v23);
	\draw(v20) -- (v23);
	\draw(v21) -- (v23);
	\draw(v22) -- (v23);
	\draw(v23) -- (v24);
	\draw(3.25*\x,.8*\y) node{\sOne{1}{\rs}};
	\draw(2.25*\x,1.8*\y) node{\sTwo{1}{\rs}};
	\draw(3.25*\x,1.8*\y) node{\sThree{1}{\rs}};
	\draw(4.75*\x,1.8*\y) node{\sFour{1}{\rs}};
	\draw(2.25*\x,3.2*\y) node{\sFive{1}{\rs}};
	\draw(3.75*\x,3.2*\y) node{\sSix{1}{\rs}};
	\draw(4.75*\x,3.2*\y) node{\sSeven{1}{\rs}};
	\draw(4*\x,4*\y) node{\sEight{1}{\rs}};
	\draw(1.75*\x,4.8*\y) node{\sNine{1}{\rs}};
	\draw(2.7*\x,4.8*\y) node{\sTen{1}{\rs}};
	\draw(4.3*\x,4.8*\y) node{\sEleven{1}{\rs}};
	\draw(5.3*\x,4.8*\y) node{\sTwelve{1}{\rs}};
	\draw(1*\x,6.4*\y) node{\sThirteen{1}{\rs}};
	\draw(1.75*\x,6*\y) node{\sFourteen{1}{\rs}};
	\draw(2.75*\x,6.05*\y) node{\sFifteen{1}{\rs}};
	\draw(4.35*\x,5.95*\y) node{\sSixteen{1}{\rs}};
	\draw(5.25*\x,6*\y) node{\sSeventeen{1}{\rs}};
	\draw(6*\x,6.4*\y) node{\sEighteen{1}{\rs}};
	\draw(1.75*\x,7.2*\y) node{\sNineteen{1}{\rs}};
	\draw(2.8*\x,7.2*\y) node{\sTwenty{1}{\rs}};
	\draw(4.2*\x,7.2*\y) node{\sTwentyone{1}{\rs}};
	\draw(5.25*\x,7.2*\y) node{\sTwentytwo{1}{\rs}};
	\draw(3.75*\x,8.2*\y) node{\sTwentythree{1}{\rs}};
	\draw(3.75*\x,8.75*\y) node{\sTwentyfour{1}{\rs}};
	\begin{pgfonlayer}{background}
		\draw(3.5*\x,1*\y) node[circle,fill=black!10!white,scale=2*\r]{};
		\draw(2.5*\x,2*\y) node[circle,fill=black!10!white,scale=2*\r]{};
		\draw(3.5*\x,2*\y) node[circle,fill=black!10!white,scale=2*\r]{};
		\draw(4.5*\x,2*\y) node[circle,fill=black!10!white,scale=2*\r]{};
		\draw(2.5*\x,3*\y) node[circle,fill=black!10!white,scale=2*\r]{};
		\draw(3.5*\x,3*\y) node[circle,fill=black!10!white,scale=2*\r]{};
		\draw(4.5*\x,3*\y) node[circle,fill=black!10!white,scale=2*\r]{};
		\draw(3.5*\x,4*\y) node[circle,fill=black!10!white,scale=2*\r]{};
 		\draw(2*\x,5*\y) node[circle,fill=black!10!white,scale=2*\r](q9){};
		  \draw[line width=15pt,draw=black!10!white](3.5*\x,4*\y) -- (2*\x,5*\y);
		\draw(3*\x,5*\y) node[circle,fill=black!10!white,scale=2*\r]{};
		\draw(1*\x,6*\y) node[circle,fill=black!10!white,scale=2*\r]{};
		  \draw[line width=15pt,draw=black!10!white](3*\x,5*\y) -- (1*\x,6*\y);
		\draw(4*\x,5*\y) node[circle,fill=black!10!white,scale=2*\r]{};
		\draw(2*\x,6*\y) node[circle,fill=black!10!white,scale=2*\r]{};
		  \draw[line width=15pt,draw=black!10!white](4*\x,5*\y) -- (2*\x,6*\y);
		\draw(5*\x,5*\y) node[circle,fill=black!10!white,scale=2*\r]{};
		\draw(3*\x,6*\y) node[circle,fill=black!10!white,scale=2*\r]{};
		  \draw[line width=15pt,draw=black!10!white](5*\x,5*\y) -- (3*\x,6*\y);
		\draw(4*\x,6*\y) node[circle,fill=black!10!white,scale=2*\r]{};
		\draw(2*\x,7*\y) node[circle,fill=black!10!white,scale=2*\r]{};
		  \draw[line width=15pt,draw=black!10!white](4*\x,6*\y) -- (2*\x,7*\y);
		\draw(5*\x,6*\y) node[circle,fill=black!10!white,scale=2*\r]{};
		\draw(3*\x,7*\y) node[circle,fill=black!10!white,scale=2*\r]{};
		  \draw[line width=15pt,draw=black!10!white](5*\x,6*\y) -- (3*\x,7*\y);
		\draw(6*\x,6*\y) node[circle,fill=black!10!white,scale=2*\r]{};
		\draw(4*\x,7*\y) node[circle,fill=black!10!white,scale=2*\r]{};
		  \draw[line width=15pt,draw=black!10!white](6*\x,6*\y) -- (4*\x,7*\y);
		\draw(5*\x,7*\y) node[circle,fill=black!10!white,scale=2*\r]{};
		\draw(3.5*\x,8*\y) node[circle,fill=black!10!white,scale=2*\r]{};
		\draw(3.5*\x,8.75*\y) node[circle,fill=black!10!white,scale=2*\r]{};
		  \draw[line width=15pt,draw=black!10!white](5*\x,7*\y) -- (3.5*\x,8*\y) -- (3.5*\x,8.75*\y);
	\end{pgfonlayer}
\end{tikzpicture}
	\caption{The lattice of proper mergings of a $3$-star and a $1$-chain, where the nodes are labeled
	  with the corresponding proper merging. The $1$-chain is represented by the black node, and the 
	  $3$-star by the (labeled) white nodes. The highlighted edges and vertices indicate the congruence
	  classes with respect to the lattice homomorphism $\eta$ defined in \eqref{eq:eta}.}
	\label{fig:mergings_chain_star}
\end{figure}

\section{Enumerating Proper Mergings of Stars and Chains}
  \label{sec:enumeration}
In order to enumerate the proper mergings of an $m$-star and an $n$-chain, we investigate a decomposition of
the set of proper mergings of an $m$-antichain and an $n$-chain, and determine for every 
$(R,T)\in\mathfrak{A\!C}_{m,n}^{\bullet}$ the number of elements in the fiber $\eta^{-1}(R,T)$.

\subsection{Decomposing the Set $\mathfrak{A\!C}_{m,n}^{\bullet}$}
  \label{sec:decomposition}
Denote by $\mathfrak{A\!C}_{m,n}^{\bullet}(k_{1},k_{2})$ the set of proper mergings 
$(R,T)\in\mathfrak{A\!C}_{m,n}^{\bullet}$ satisfying the following condition: $k_{1}$ is the minimal index 
such that there exists some $j_{1}\in\{1,2,\ldots,m\}$ with $a_{j_{1}}\;R\;c_{k_{1}}$, and $k_{2}$ is the 
maximal index such that there exists some $j_{2}\in\{1,2,\ldots,m\}$ with $c_{k_{2}}\;T\;a_{j_{2}}$. By 
convention, if $R=\emptyset$, then we set $k_{1}:=n+1$, and if $T=\emptyset$, then we set $k_{2}:=0$. Let 
$\biguplus$ denote the disjoint set union.

\begin{lemma}
  \label{lem:decomposition}
	If $(R,T)\in\mathfrak{A\!C}_{m,n}^{\bullet}(k_{1},k_{2})$ is a proper merging of $\ac$ and $\cc$,
	then $k_{1}>k_{2}$. Moreover we have
	\begin{displaymath}
		\mathfrak{A\!C}_{m,n}^{\bullet}=
		  \biguplus_{k_{1}=1}^{n+1}\biguplus_{k_{2}=0}^{k_{1}-1}
		  {\mathfrak{A\!C}_{m,n}^{\bullet}(k_{1},k_{2})}.
	\end{displaymath}
\end{lemma}
\begin{proof}
	Let $(R,T)\in\mathfrak{A\!C}_{m,n}^{\bullet}$. Denote by $\leq_{R,T}$ the order relation induced by 
	the proper merging $(R,T)$ on the set $A\cup C$. Assume that $k_{1}\leq k_{2}$. This means that 
	there exist elements $a_{j_{1}},a_{j_{2}}\in A$ with $a_{j_{1}}\leq_{R,T}c_{k_{1}}$ and 
	$c_{k_{2}}\leq_{R,T}a_{j_{2}}$. If $k_{1}=k_{2}$, then $c_{k_{1}}=c_{k_{2}}$, and this implies that 
	$a_{j_{1}}=a_{j_{2}}$ (since $\ac$ is an antichain) which is a contradiction to $(R,T)$ being a 
	proper merging. If $k_{1}<k_{2}$, we have $c_{k_{1}}<c_{k_{2}}$, and thus 
	$a_{j_{1}}\leq_{R,T}c_{k_{1}}<c_{k_{2}}\leq_{R,T}a_{j_{2}}$. This is a contradiction to 
	$R\circ T$ being contained in $=_{\ac}$.
	
	It is clear that the values $k_{1}$ and $k_{2}$ are uniquely determined, and thus the result follows.
\end{proof}

For later use, we will decompose $\mathfrak{A\!C}_{m,n}^{\bullet}(k_{1},k_{2})$ even further. Let 
$(R,T)\in\mathfrak{A\!C}_{m,n}^{\bullet}(k_{1},k_{2})$. It is clear that there exists a maximal index 
$l\in\{0,1,\ldots,k_{2}\}$ such that $c_{l}\;R\;a$ for all $a\in A$. (The case $l=0$ is to be interpreted
as the case where there exists no $c_{l}$ with the desired property.) Denote by 
$\mathfrak{A\!C}_{m,n}^{\bullet}(k_{1},k_{2},l)$ the set of proper mergings 
$(R,T)\in\mathfrak{A\!C}_{m,n}^{\bullet}(k_{1},k_{2})$ with $l$ being the maximal index such that 
$c_{l}\;R\;a_{j}$ for all $j\in\{1,2,\ldots,m\}$. Similarly to Lemma~\ref{lem:decomposition}, we can show 
that
\begin{displaymath}
	\mathfrak{A\!C}_{m,n}^{\bullet}(k_{1},k_{2})=
	  \biguplus_{l=0}^{k_{2}}{\mathfrak{A\!C}_{m,n}^{\bullet}(k_{1},k_{2},l)},
\end{displaymath}
and we obtain
\begin{equation}
  \label{eq:cardinality_decomposition_1}
	\Bigl\lvert\mathfrak{A\!C}_{m,n}^{\bullet}\Bigr\rvert=\sum_{k_{1}=1}^{n+1}\sum_{k_{2}=0}^{k_{1}-1}
	  \sum_{l=0}^{k_{2}}{\Bigl\lvert\mathfrak{A\!C}_{m,n}^{\bullet}(k_{1},k_{2},l)\Bigr\rvert}.
\end{equation}

\subsection{Determining the Cardinality of $\mathfrak{A\!C}^{\bullet}_{m,n}(k_{1},k_{2},l)$}
  \label{sec:cardinality_parts}
In \cite{muehle12counting}*{Section~5}, the author has investigated the number of proper mergings of 
an $m$-antichain and an $n$-chain, and has constructed a bijection from these proper mergings to monotone
$(n+1)$-colorings of the complete bipartite graph $\vec{K}_{m,m}$. This bijection is essential for 
determining the cardinality of $\mathfrak{A\!C}^{\bullet}_{m,n}(k_{1},k_{2},l)$. Let us recall this 
construction briefly. Let $V$ be the vertex set of a complete bipartite digraph $\vec{K}_{m,m}$ partitioned 
into sets $V_{1}$ and $V_{2}$ such that $\lvert V_{1}\rvert=\lvert V_{2}\rvert=m$, and such that the set 
$\vec{E}$ of edges of $\vec{K}_{m,m}$ satisfies $\vec{E}=V_{1}\times V_{2}$. A 
monotone $n$-coloring of $\vec{K}_{m,m}$ is now a map $\gamma:V\to\{1,2,\ldots,n\}$ satisfying 
$\gamma(v)\leq\gamma(v')$ for all $(v,v')\in\vec{E}$.

Let $\ac$ and $\cc$ denote 
an $m$-antichain and an $n$-chain, respectively, as defined in Sections~\ref{sec:m_stars} and 
\ref{sec:n_chains}. For $(R,T)\in\mathfrak{A\!C}_{m,n}^{\bullet}$, we construct a coloring $\gamma_{(R,T)}$ 
of $\vec{K}_{m,m}$ as follows
\begin{equation}
  \label{eq:coloring}
	\gamma_{(R,T)}(v_{i})=n+1-k\quad\mbox{if and only if}\quad
	  \begin{cases}
		v_{i}\in V_{1} & \mbox{and}\;a_{i}\;R\;c_{j}\;\mbox{for all}\\ 
		  & j\in\{k+1,k+2,\ldots,n\},\\
		v_{i}\in V_{2} & \mbox{and}\;c_{j}\;T\;a_{i}\;\mbox{for all}\\
		  & j\in\{1,2,\ldots,k\},
	  \end{cases}
\end{equation}
where $i\in\{1,2,\ldots,m\}$. It is the statement of \cite{muehle12counting}*{Theorem~5.6} that this defines a 
bijection between $\mathfrak{A\!C}_{m,n}^{\bullet}$ and the set of monotone $(n+1)$-colorings of $\vec{K}_{m,m}$. 
The next lemma describes how the monotone $(n+1)$-coloring of $\vec{K}_{m,m}$ induced by $(R,T)$ is influenced by 
the parameters $k_{1},k_{2}$ and $l$.

\begin{lemma}
  \label{lem:restricted_coloring}
	Let $(R,T)\in\mathfrak{A\!C}_{m,n}^{\bullet}(k_{1},k_{2},l)$. The monotone $(n+1)$-coloring
	$\gamma_{(R,T)}$ of $\vec{K}_{m,m}$ as defined in \eqref{eq:coloring} satisfies 
	\begin{align*}
		1\leq\gamma_{(R,T)}(v) & \leq n+2-k_{1}\quad\mbox{if}\;v\in V_{1},\qquad\mbox{and}\\
		n+1-l\geq\gamma_{(R,T)}(v) & \geq n+1-k_{2}\quad\mbox{if}\;v\in V_{2},
	\end{align*}
	and there is at least one vertex $v^{(1)}\in V_{1}$ with 
	$\gamma_{(R,T)}\bigl(v^{(1)}\bigr)=n+2-k_{2}$, and there is at least one vertex $v^{(2)}\in V_{2}$ 
	with $\gamma_{(R,T)}\bigl(v^{(2)}\bigr)=n+1-k_{2}$, and at least one vertex 
	$v'^{(2)}\in V_{2}$ with $\gamma_{(R,T)}\bigl(v'^{(2)}\bigr)=n+1-l$.
\end{lemma}
\begin{proof}
	Assume that there exists some $t\in\{1,2,\ldots,m\}$ such that the vertex 
	$v_{t}\in V_{1}$ satisfies $\gamma_{(R,T)}(v_{t})=k>n+2-k_{1}$. In view of \eqref{eq:coloring}, this
	means that $a_{t}\;R\;c_{j}$ for all $j\in\{n+2-k,n+3-k,\ldots,n\}$, in particular 
	$a_{t}\;R\;c_{n+2-k}$. We have $n+2-k<n+2-(n+2-k_{1})=k_{1}$, and thus 
	$c_{n+2-k}<_{\cc}c_{k_{1}}$ which contradicts the minimality of $k_{1}$. If all $v\in V_{1}$ 
	have $\gamma_{(R,T)}(v)\leq n+2-k_{1}$, then we obtain a contradiction to the minimality of $k_{1}$ 
	in an analogous way. The argument for the vertices in $V_{2}$ works similar. Note that we have
	to consider both bounds $k_{2}$ and $l$.
\end{proof}

The next two lemmas determine the cardinality of $\mathfrak{A\!C}_{m,n}^{\bullet}(k_{1},k_{2},l)$ for every
valid triple $(k_{1},k_{2},l)$ by enumerating the corresponding monotone colorings of $\vec{K}_{m,m}$. Note
that the number of possible ways to color the vertex set $V_{1}$ depends on the parameters $m,n$ and $k_{1}$, 
while the number of possible ways to color the vertex set $V_{2}$ depend on the parameters $m,k_{2}$ and $l$. 
For a fixed choice of indices $k_{1},k_{2}$ and $l$, denote by $F_{V_{1}}(m,n,k_{1})$ the number of possible 
colorings of $V_{1}$, and denote by $F_{V_{2}}(m,k_{2},l)$ the number of possible colorings of $V_{2}$. 

\begin{lemma}
  \label{lem:count_v1}
	For $k_{1}\in\{1,2,\ldots,n+1\}$, we have 
	\begin{displaymath}
		F_{V_{1}}(m,n,k_{1})=(n+2-k_{1})^{m}-(n+1-k_{1})^{m}.
	\end{displaymath}
\end{lemma}
\begin{proof}
	Let $V=V_{1}\cup V_{2}$ be the vertex set of $\vec{K}_{m,m}$ where $V_{1},V_{2}$ are maximal disjoint
	independent sets of $\vec{K}_{m,m}$. Recall that we want to count the possible colorings of 
	$\vec{K}_{m,m}$ such that the vertices in $V_{1}$ have color at most $n+2-k_{1}$ and there is at 
	at least one vertex in $V_{1}$ having color exactly $n+2-k_{1}$.
	
	A standard counting argument shows that there are precisely $(n+2-k_{1})^{m}$ ways to color the $m$ 
	vertices of $V_{1}$ with colors in $\{1,2,\ldots,n+2-k_{1}\}$. Since we require that at least one 
	vertex has color $n+2-k_{1}$, we have to exclude the cases where every vertex is colored 
	$\leq n+1-k_{1}$. The same counting argument shows that there are $(n+1-k_{1})^{m}$-many such 
	colorings. Hence the number of ways to color the vertices of $V_{1}$ with the given restrictions is
	precisely $(n+2-k_{1})^{m}-(n+1-k_{1})^{m}$ as desired.
\end{proof}

\begin{lemma}
  \label{lem:count_v2}
	Let $k_{1}\in\{1,2,\ldots,n+1\}$. For $k_{2}\in\{0,1,\ldots,k_{1}-1\}$ and 
	$l\in\{0,1,\ldots,k_{2}\}$, we have
	\begin{displaymath}
		F_{V_{2}}(m,k_{2},l)=
			\begin{cases}
				1, & k_{2}=l,\;\mbox{or}\\
				(k_{2}-l+1)^{m}-2(k_{2}-l)^{m}+(k_{2}-l-1)^{m}, & \mbox{otherwise}.
			\end{cases}
	\end{displaymath}
\end{lemma}
\begin{proof}
	Let $V=V_{1}\cup V_{2}$ be the vertex set of $\vec{K}_{m,m}$ where $V_{1},V_{2}$ are maximal disjoint
	independent sets of $\vec{K}_{m,m}$. Recall that we want to count the possible colorings of 
	$\vec{K}_{m,m}$ such that the vertices in $V_{2}$ have colors in 
	$\{n+1-k_{2},n+2-k_{2},\ldots,n+1-l\}$ with at least one vertex having color exactly $n+1-k_{2}$, 
	and at least one vertex having color exactly $n+1-l$. 
	
	If $k_{2}=l$, it follows from Lemma~\ref{lem:restricted_coloring} that every vertex in $V_{2}$
	has color $n+1-k_{2}=n+1-l$. There is obviously only one possibility.
	
	So let $l<k_{2}$. With the same standard counting argument as in the proof of the previous lemma, we 
	notice that there are precisely $(k_{2}-l+1)^{m}$ ways to color the $m$ vertices of $V_{2}$ with 
	colors in $\{n+1-k_{2},n+2-k_{2},\ldots,n+1-l\}$. Since we require to color at least one vertex with 
	color $n+1-k_{2}$ and at least one vertex with color $n+1-l$, we have to subtract the cases where 
	all vertices have color $\geq n+2-k_{2}$ and the cases where all vertices have color $\leq n-l$. 
	However, we subtract the cases where all vertices have a color in 
	$\{n+2-k_{2},n+3-k_{2},\ldots,n-l\}$ twice, so we have to add these again. Thus, with an analogous 
	counting argument as before, we obtain
	\begin{displaymath}
		F_{V_{2}}(m,k_{2},l)=(k_{2}-l+1)^{m}-2(k_{2}-l)^{m}+(k_{2}-l-1)^{m},
	\end{displaymath}
	as desired.
\end{proof}

Every proper merging in $\mathfrak{A\!C}_{m,n}^{\bullet}(k_{1},k_{2},l)$ corresponds to a monotone coloring
of $\vec{K}_{m,m}$ where the colors respect the restrictions described in 
Lemma~\ref{lem:restricted_coloring}. Since $k_{1}>k_{2}$ (see Lemma~\ref{lem:decomposition}) we notice that
the largest possible color for $V_{1}$ is strictly smaller than the smallest possible color for $V_{2}$,
and we obtain
\begin{equation}
  \label{eq:cardinality_decomposition_2}
	\Bigl\lvert\mathfrak{A\!C}_{m,n}^{\bullet}(k_{1},k_{2},l)\Bigr\rvert
	  = F_{V_{1}}(m,n,k_{1})\cdot F_{V_{2}}(m,k_{2},l).
\end{equation}

\subsection{Determining the Cardinality of the Fibers}
  \label{sec:cardinality_fibers}
We have seen in Section~\ref{sec:embedding} that $\bigl(\mathfrak{A\!C}_{m,n}^{\bullet},\preceq\bigr)$ is a
quotient lattice of $\bigl(\mathfrak{S\!C}_{m,n}^{\bullet},\preceq\bigr)$. Thus, every proper merging of an 
$m$-antichain and an $n$-chain corresponds to a set of proper mergings of an $m$-star and an $n$-chain
(namely the corresponding fiber under the lattice homomorphism $\eta$), and these sets are pairwise disjoint.
Thus, if we can determine the number of elements in each fiber, then we can determine the number of all 
proper mergings of an $m$-star and an $n$-chain.

Let $(R,T)\in\mathfrak{S\!C}_{m,n}^{\bullet}$ be a proper merging of an $m$-star and an $n$-chain. In the 
following, we write for some $j\in\{1,2,\ldots,n\}$ simply ``$s_{0}\leq_{R,T}c_{j}$'' to mean that we create
a pair of relations $(R',T)$ from $(R,T)$ by setting
\begin{displaymath}
	R'=R\cup\bigl\{(s_{0},c_{j}),(s_{0},c_{j+1}),\ldots,(s_{0},c_{n})\bigr\}.
\end{displaymath}
Similarly, we write ``$c_{j}\leq_{R,T}s_{0}$'' for some $j\in\{1,2,\ldots,n\}$ to mean that we create a 
new pair of relations $(R,T')$ from $(R,T)$ by setting
\begin{displaymath}
	T'=T\cup\bigl\{(c_{1},s_{i}),(c_{2},s_{i}),\ldots,(c_{j},s_{i})\bigr\},
	  \;\mbox{for all}\;i\in\{0,1,\ldots,m\}.
\end{displaymath}
For $c\in C$, the operations ``$s_{0}\leq_{R,T}c$'' respectively ``$c\leq_{R,T}s_{0}$'' can be understood 
as adding a covering relation $(s_{0},c)$ respectively $(c,s_{0})$ to the proper merging $(R,T)$ and 
applying transitive closure. Thus, it is not immediately clear that these operations yield a merging of an 
$m$-star and an $n$-chain at all. The next lemma determines the number of \emph{proper} mergings we can 
generate from the image under the map $\xi$ of a proper merging of an $m$-antichain and an $n$-chain.

\begin{lemma}
  \label{lem:count_valid}
	Let $(R,T)\in\mathfrak{A\!C}_{m,n}^{\bullet}(k_{1},k_{2},l)$. Then 
	$\bigl\lvert\eta^{-1}(R,T)\bigr\rvert=k_{1}(l+1) - \binom{l+1}{2}$.
\end{lemma}
\begin{proof}
	By construction, we have $\xi(R,T)=(R_{o},T_{o})\in\eta^{-1}(R,T)$, and 
	$s_{0}\leq_{R_{o},T_{o}}c_{k}$ for all $k\in\{k_{1},k_{1}+1,\ldots,n\}$. Thus, performing 
	``$s_{0}\leq_{R_{o},T_{o}}c_{j}$'' for some $j\geq k_{1}$ would simply do nothing. 
	Performing ``$c_{j}\leq_{R_{o},T_{o}}s_{0}$'' for some $j\geq k_{1}$ adds in particular the relation 
	$(c_{j},a_{k})$ to $T_{o}$ for all $k\in\{1,2,\ldots,m\}$. Since 
	$(R,T)\in\mathfrak{A\!C}_{m,n}^{\bullet}(k_{1},k_{2},l)$, we can assume that there exists some
	$i\in\{1,2,\ldots,m\}$ such that $a_{i}\;R_{o}\;c_{k_{1}}$, and thus in particular 
	$a_{i}\;R_{o}\;c_{j}$. Thus we have $c_{j}\;T_{o}'\;a_{i}$ and $a_{i}\;R\;c_{j}$, which is a 
	contradiction to $(R,T')$ being a proper merging.	
	Hence, we can only create new proper mergings from $(R_{o},T_{o})$ by applying the operations
	``$s_{0}\leq_{R_{o},T_{o}}c_{j}$'' or ``$c_{j}\leq_{R_{o},T_{o}}s_{0}$'' for some 
	$j\in\{1,2,\ldots,k_{1}-1\}$.
	
	If we perform ``$c_{j}\leq_{R_{o},T_{o}}s_{0}$'' for some $j\in\{k_{2}+1,k_{2}+2,\ldots,k_{1}-1\}$,
	then we obtain a proper merging $(R_{o},T_{o}')$ which contains the relations $(c_{j},a_{i})$ for
	all $i\in\{1,2,\ldots,m\}$. Hence, $\eta(R_{o},T_{o}')\neq(R,T)$, and thus 
	$(R_{o},T_{o}')\notin\eta^{-1}(R,T)$. However, we can perform ``$s_{0}\leq_{R_{o},T_{o}}c_{j}$''
	for every $j\in\{k_{2}+1,k_{2}+2,\ldots,k_{1}-1\}$ without problems. This gives us 
	$(k_{1}-k_{2}-1)$-many new proper mergings in $\eta^{-1}(R,T)$.
	
	With the same reasoning as before, we see that performing ``$c_{j}\leq_{R_{o},T_{o}}s_{0}$'' for 
	some $j\in\{l+1,l+2,\ldots,k_{2}\}$ yields a proper merging $(R_{o},T_{o}')\notin\eta^{-1}(R,T)$,
	but we can apply ``$s_{0}\leq_{R_{o},T_{o}}c_{j}$'' for every such $j$, giving us 
	$(k_{2}-l)$-many new proper mergings in $\eta^{-1}(R,T)$.
	
	Now let $j\in\{1,2,\ldots,l\}$. Performing ``$c_{j}\leq_{R_{o},T_{o}}s_{0}$'' works fine in this 
	case, and we obtain a proper merging $(R_{o},T_{o}')$. Additionally, we can now perform
	``$s_{0}\leq_{R_{o},T_{o}'}c_{i}$'' for every $i\in\{j+1,j+2,\ldots,k_{1}-1\}$ to obtain a new 
	proper merging from $(R_{o},T_{o}')$. Note the new
	subscript ``$R_{o},T_{o}'$'' in the operator! (Suppose that we perform 
	``$s_{0}\leq_{R_{o},T_{o}'}c_{i}$'' for some $i\in\{1,2,\ldots,j\}$. Then we had
	$s_{0}\;R_{o}'\;c_{i}\;T_{o}'\;s_{0}$ which is a contradiction to $(R_{o}',T_{o}')$ being a proper 
	merging. Performing ``$s_{0}\leq_{R_{o},T_{o}'}c_{i}$'' for some $i\in\{k_{1},k_{1}+1,\ldots,n\}$
	would yield $(R_{o}',T_{o}')=(R_{o},T_{o}')$.) Thus, for every $j\in\{1,2,\ldots,l\}$ we obtain
	$(k_{1}-j)$-many new proper mergings in $\eta^{-1}(R,T)$. Finally, we can also perform 
	``$s_{0}\leq_{R_{o},T_{o}}c_{j}$'' to obtain a new proper merging $(R_{o}',T_{o})\in\eta^{-1}(R,T)$. 
	However, we cannot perform ``$c_{i}\leq_{R_{o}',T_{o}}s_{0}$'' for any $i\in\{1,2,\ldots,n\}$, 
	because we would either obtain a contradiction or a proper merging we have already counted. Hence,
	this case gives us $l$ new proper mergings in $\eta^{-1}(R,T)$. 
	
	Now we just have to add all the possibilities and obtain
	\begin{align*}
		\bigl\lvert\eta^{-1}(R,T)\bigr\rvert 
		  & = 1 + (k_{1}-k_{2}-1) + (k_{2}-l) + \Bigl(\sum_{j=1}^{l}{k_{1}-j}\Bigr) + l\\
		& = k_{1}(l+1) - \frac{l(l+1)}{2}\\
		& = k_{1}(l+1) - \binom{l+1}{2},
	\end{align*}
	as desired.
\end{proof}

Now we are set to enumerate the proper mergings of an $m$-star and an $n$-chain. 

\begin{lemma}
  \label{lem:equality}
	For $m,n\in\mathbb{N}$, we have $F_{\fs\!\cc}(m,n)=C(m,n+1)$, where $C$ is defined in 
	\eqref{eq:count}.
\end{lemma}
\begin{proof}
	Putting \eqref{eq:cardinality_decomposition_1}, \eqref{eq:cardinality_decomposition_2} and 
	Lemmas~\ref{lem:count_v1}--\ref{lem:count_valid} together, we obtain
	\begin{align}\label{eq:init}
		F_{\fs\!\cc}(m,n) & =\sum_{(R,T)\in\mathfrak{A\!C}_{m,n}^{\bullet}}
		  {\bigl\lvert\eta^{-1}(R,T)\bigr\rvert}\\
\nonumber	& = \sum_{k_{1}=1}^{n+1}\sum_{k_{2}=0}^{k_{1}-1}\sum_{l=0}^{k_{2}}
		  \sum_{(R,T)\in\mathfrak{A\!C}_{m,n}^{\bullet}(k_{1},k_{2},l)}
		  {\bigl\lvert\eta^{-1}(R,T)\bigr\rvert}\\  
\nonumber	& = \sum_{k_{1}=1}^{n+1}\sum_{k_{2}=0}^{k_{1}-1}\sum_{l=0}^{k_{2}}
		  F_{V_{1}}(m,n,k_{1})\cdot F_{V_{2}}(m,k_{2},l)\cdot\left(k_{1}(l+1)-\binom{l+1}{2}\right)\\
\nonumber	& = \sum_{k_{1}=1}^{n+1}{F_{V_{1}}(m,n,k_{1})\sum_{k_{2}=0}^{k_{1}-1}\sum_{l=0}^{k_{2}}
		  F_{V_{2}}(m,k_{2},l)\cdot\left(k_{1}(l+1)-\binom{l+1}{2}\right)}.
	\end{align}
	The proof that this last sum equals $C(m,n+1)$ is not very difficult, but rather technical and longish. 
	Thus we have decided to provide this proof in every detail in Appendix~\ref{app:proof}.
\end{proof}

\begin{proof}[Proof of Theorem~\ref{thm:cardinality}]
	This follows from Lemma~\ref{lem:equality}.
\end{proof}

\begin{remark}
	The presented proof of Theorem~\ref{thm:cardinality} is obtained by counting the proper mergings
	of an $m$-star and an $n$-chain in a rather na\"ive way, and the conversion of the na\"ive counting
	formula into the desired formula is rather longish. Christian Krattenthaler proposed a family
	of objects that are also counted by $C(m,n+1)$: let $V_{1},V_{2}$, and $V_{3}$ be 
	disjoint sets with cardinalities $\lvert V_{1}\rvert=k_{1},\lvert V_{2}\rvert=k_{2}$, and 
	$\lvert V_{3}\rvert=k_{3}$, and denote by $\vec{K}_{k_{1},k_{2},k_{3}}$ the directed graph 
	$(V,\vec{E})$ whose vertex set is $V=V_{1}\cup V_{2}\cup V_{3}$, and whose set of edges is 
	$\vec{E}=(V_{1}\times V_{2})\cup(V_{2}\times V_{3})$. A monotone $(n+1)$-coloring of a directed graph
	is an assignment of at most $n+1$ different numbers to the vertices of the graph such that the numbers weakly 
	increase along directed edges. A standard counting argument shows that the number of monotone 
	$(n+1)$-colorings of $\vec{K}_{m+1,1,m}$ is precisely $C(m,n+1)$. A much more elegant, and perhaps much 
	simpler proof of Theorem~\ref{thm:cardinality} could thus be obtained by solving the following problem.
\end{remark}

\begin{problem}[Solved by Jonathan Farley, May 2013]
  \label{prob:bijection}
	Construct a bijection between the set $\mathfrak{S\!C}_{m,n}^{\bullet}$ of proper mergings of an 
	$m$-star and an $n$-chain, and the set $\Gamma_{n+1}(\vec{K}_{m+1,1,m})$ of monotone $(n+1)$-colorings
	of $\vec{K}_{m+1,1,m}$.
\end{problem}

\subsection{\textsc{Update}: Jonathan Farley's Solution of Problem~\ref{prob:bijection}}
  \label{sec:bijection_solution}
Recently, Jonathan Farley \cite{farley13bijection} has solved Problem~\ref{prob:bijection}. We will briefly explain
his bijection in this section. Let $\PP=(P,\leq_{\PP})$ and $\QQ=(Q,\leq_{\QQ})$ be two posets. We say that $\PP$ is 
\alert{bounded} if it has a unique minimal element, denoted by $0_{\PP}$, and a unique maximal element, denoted by 
$1_{\PP}$, and likewise for $Q$. The \alert{ordinal sum of $\PP$ and $\QQ$} is the poset $\PP\oplus\QQ=(P\cup Q,\leq)$, 
with $p\leq q$ if and only if either (i) $p,q\in P$ and $p\leq_{\PP}q$, (ii) $p,q\in Q$ and $p\leq_{\QQ}q$, or (iii) 
$p\in P$ and $q\in Q$. If $\PP$ has a unique maximal element $1_{\PP}$, and if $\QQ$ has a unique minimal element
$0_{\QQ}$, then the \alert{coalesced ordinal sum of $\PP$ and $\QQ$} is the ordinal sum of $\PP$ and $\QQ$ with 
$1_{\PP}$ and $0_{\QQ}$ identified, and will be denoted by $\PP\oplus_{c}\QQ$. Now let $\zeta:\PP\to\QQ$ be a map from 
$\PP$ to $\QQ$. We say that $\zeta$ is \alert{order-preserving} if $p\leq_{\PP}p'$ implies $\zeta(p)\leq_{\QQ}\zeta(p')$. 
Further, if $\PP$ and $\QQ$ are bounded, we say that $\zeta$ is \alert{bound-preserving} if $\zeta(0_{\PP})=0_{\QQ}$ and 
$\zeta(1_{\PP})=0_{\QQ}$. 
	
Let $\ac_{k}$ denote an antichain with $k$ elements, let $\cc_{k}$ denote a chain with $k$ elements, and let 
$\BB_{k}$ denote the Boolean lattice with $2^{k}$ elements. For two posets $\PP$ and $\QQ$, let $O\!P(\PP,\QQ)$ 
denote the set of order-preserving maps from $\PP$ to $\QQ$, and if $\PP$ and $\QQ$ are lattices, let $B\!P(\PP,\QQ)$ 
denote the set of bound-preserving lattice homomorphisms from $\PP$ to $\QQ$. 

It is easy to see that 
\begin{displaymath}
	\Gamma_{n+1}(\vec{K}_{m+1,1,m})\cong O\!P\bigl(\ac_{m+1}\oplus\ac_{1}\oplus\ac_{m},\cc_{n+1}\bigr).
\end{displaymath}

Using Priestley's Representation Theorem For Distributive Lattices, see for instance 
\cite{davey02introduction}*{Theorem~11.23}, we conclude that 
\begin{displaymath}
	\left\lvert O\!P\bigl(\ac_{m+1}\oplus\ac_{1}\oplus\ac_{m},\cc_{n+1}\bigr)\right\rvert
	  = \left\lvert B\!P\bigl(\cc_{n+2},\BB_{m}\oplus_{c}\BB_{1}\oplus_{c}\BB_{m+1}\bigr)\right\rvert.
\end{displaymath}
Since $\cc_{n+2}$ is a chain, we find 
\begin{displaymath}
	\left\lvert B\!P\bigl(\cc_{n+2},\BB_{m}\oplus_{c}\BB_{1}\oplus_{c}\BB_{m+1}\bigr)\right\rvert
	  = \left\lvert O\!P\bigl(\cc_{n+2},\BB_{m}\oplus_{c}\BB_{1}\oplus_{c}\BB_{m+1}\bigr)\right\rvert,
\end{displaymath}
and if we forget about the bounds, we obtain
\begin{displaymath}
	\left\lvert O\!P\bigl(\cc_{n+2},\BB_{m}\oplus_{c}\BB_{1}\oplus_{c}\BB_{m+1}\bigr)\right\rvert
	  = \left\lvert O\!P\bigl(\cc_{n},(\BB_{m}\setminus\{0_{\BB_{m}}\})\oplus\BB_{0}\oplus\BB_{m+1})\right\rvert.
\end{displaymath}
We notice that order-preserving maps from a chain to a poset $\PP$ are in bijection with multichains of $\PP$. Clearly,
to every multichain in $\PP\oplus\QQ$, we can associate a unique multichain in $\QQ\oplus\PP$, by exchanging the 
corresponding components. Hence, the order of the summands does not really play a role, and we obtain 
\begin{displaymath}
	\left\lvert O\!P\bigl(\cc_{n},(\BB_{m}\setminus\{0_{\BB_{m}}\})\oplus\BB_{0}\oplus\BB_{m+1}\right\rvert
	  = \left\lvert O\!P\bigl(\cc_{n},\BB_{0}\oplus\BB_{m+1}\oplus_{c}\BB_{m})\right\rvert.
\end{displaymath}

The next step is to construct a bijection from $O\!P\bigl(\cc_{n},\BB_{0}\oplus\BB_{m+1}\oplus_{c}\BB_{m})$ to 
$\mathfrak{S\!C}_{m,n}^{\bullet}$. For that, let $X=\{x\}$, let $Y$ be a poset which is order-isomorphic to the 
Boolean lattice whose elements are subsets of $\{0,1,2,\ldots,m\}$ (via the map $\varphi_{Y}$), and let
$Z$ be a poset which is order-isomorphic to the Boolean lattice whose elements are subsets of $\{1,2,\ldots,m\}$ 
(via the map $\varphi_{Z}$), and let $\PP_{1,m+1,m}=X\oplus Y\oplus_{c}Z$. Now let 
$\zeta\in O\!P\bigl(\cc_{n},\PP_{1,m+1,m})$, and suppose that $\zeta(c_{i})=d_{i}$ for all $i\in\{1,2,\ldots,n\}$. 
Define a proper merging $(R,T)_{\zeta}\in\mathfrak{S\!C}_{m,n}^{\bullet}$ as follows:
\begin{enumerate}
	\item[(1)] if $d_{i}=x$, then $(c_{i},s)\in T$ for all $s\in S$,
	\item[(2a)] if $d_{i}\in Y\setminus Z$, then $(s_{0},c_{i})\in R$ if and only if $0\in\varphi_{Y}(d_{i})$,
	\item[(2b)] if $d_{i}\in Y\setminus Z$, then $(c_{i},s_{j})\in T$ for all $j\in\{1,2,\ldots,k\}$ 
	  if and only if $j\notin\varphi_{Y}(d_{i})$,\quad and
	\item[(3)] if $d_{i}\in Z$, then $(s_{0},c_{i})\in R$ and $(s_{j},c_{i})\in R$ for all $j\in\{1,2,\ldots,k\}$ 
	  if and only if $j\in\varphi_{Z}(d_{i})$.
\end{enumerate}
It was shown by Farley \cite{farley13bijection} that this construction is indeed a bijection. See 
Appendix~\ref{app:illustration_bijection} for an illustration.

\section{Counting Galois Connections between Chains and Modified Boolean Lattices}
  \label{sec:galois}
In the spirit of \cite{muehle12counting}*{Sections~3.4~and~5.2}, we can use the enumeration formula for the
proper mergings of an $m$-star and an $n$-chain to determine the number of Galois connections between
$\CL(C,C,\not\geq_{\cc})$ and $\CL(S,S,\not\geq_{\fs})$. In particular, we prove the following proposition 
within this section.
\begin{proposition}
  \label{prop:galois_ballon_chain}
	Let $\fs=(S,\leq_{\fs})$ be an $m$-star and let $\cc=(C,\leq_{\cc})$ be an $n$-chain. 
	The number of Galois connections between $\CL(C,C,\not\geq_{\cc})$ and 
	$\CL(S,S,\not\geq_{\fs})$ is $\sum_{k=1}^{n+1}{k^{m}}$.
\end{proposition}
  
We have seen in Section~\ref{sec:n_chains} that $\CL(C,C,\not\geq_{\cc})$ is isomorphic to an 
$(n+1)$-chain, and the reasoning in Section~\ref{sec:m_stars} implies that $\CL(S,S,\not\geq_{\fs})$ can be 
constructed as follows: let $\mathcal{B}_{m}$ denote the Boolean 
lattice with $2^{m}$ elements. Replacing the bottom element of $\mathcal{B}_{m}$ by a $2$-chain yields 
a lattice which we call \alert{$m$-balloon}, and we denote it by $\mathcal{B}_{m}^{(1)}$.
Figure~\ref{fig:41_balloon} shows the Hasse diagram of $\mathcal{B}_{4}^{(1)}$. The labels attached to some 
of the nodes indicate how $\mathcal{B}_{4}^{(1)}$ arises as the concept lattice of the contraordinal scale 
of the $4$-star shown in Figure~\ref{fig:4_star}.

\begin{remark}
	The construction of $\mathcal{B}_{m}^{(1)}$ can be generalized easily, by replacing the bottom 
	element of $\mathcal{B}_{m}$ by an $(l+1)$-chain for some $l>1$. We call the corresponding lattice 
	an \alert{$(m,l)$-balloon}, and denote it by $\mathcal{B}_{m}^{(l)}$. However, the case $l>1$ is not 
	considered further in this article, even though it can be considered as the concept lattice of the 
	contraordinal scale of the poset that arises from an $m$-star by replacing the unique bottom element 
	by an $l$-chain. 
\end{remark}

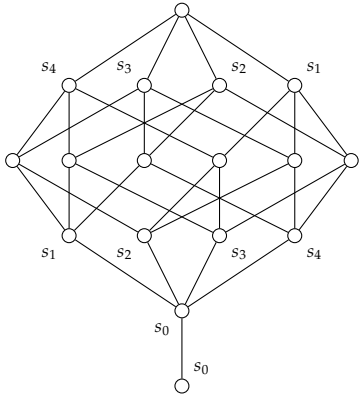
\begin{figure}[t]
	\centering
	\begin{tikzpicture}\scriptsize
		\def\x{1};
		\def\y{1};
		\draw(3.5*\x,1*\y) node[draw,circle,scale=.8,label=above right:$s_{0}$](b0){};
		\draw(3.5*\x,2*\y) node[draw,circle,scale=.8,label=below left:$s_{0}$](b1){};
		\draw(2*\x,3*\y) node[draw,circle,scale=.8,label=below left:$s_{1}$](b2){};
		\draw(3*\x,3*\y) node[draw,circle,scale=.8,label=below left:$s_{2}$](b3){};
		\draw(4*\x,3*\y) node[draw,circle,scale=.8,label=below right:$s_{3}$](b4){};
		\draw(5*\x,3*\y) node[draw,circle,scale=.8,label=below right:$s_{4}$](b5){};
		\draw(1.25*\x,4*\y) node[draw,circle,scale=.8](b6){};
		\draw(2*\x,4*\y) node[draw,circle,scale=.8](b7){};
		\draw(3*\x,4*\y) node[draw,circle,scale=.8](b8){};
		\draw(4*\x,4*\y) node[draw,circle,scale=.8](b9){};
		\draw(5*\x,4*\y) node[draw,circle,scale=.8](b10){};
		\draw(5.75*\x,4*\y) node[draw,circle,scale=.8](b11){};
		\draw(2*\x,5*\y) node[draw,circle,scale=.8,label=above left:$s_{4}$](b12){};
		\draw(3*\x,5*\y) node[draw,circle,scale=.8,label=above left:$s_{3}$](b13){};
		\draw(4*\x,5*\y) node[draw,circle,scale=.8,label=above right:$s_{2}$](b14){};
		\draw(5*\x,5*\y) node[draw,circle,scale=.8,label=above right:$s_{1}$](b15){};
		\draw(3.5*\x,6*\y) node[draw,circle,scale=.8](b16){};
		\draw(b0) -- (b1);
		\draw(b1) -- (b2);
		\draw(b1) -- (b3);
		\draw(b1) -- (b4);
		\draw(b1) -- (b5);
		\draw(b2) -- (b6);
		\draw(b2) -- (b7);
		\draw(b2) -- (b8);
		\draw(b3) -- (b6);
		\draw(b3) -- (b9);
		\draw(b3) -- (b10);
		\draw(b4) -- (b7);
		\draw(b4) -- (b9);
		\draw(b4) -- (b11);
		\draw(b5) -- (b8);
		\draw(b5) -- (b10);
		\draw(b5) -- (b11);
		\draw(b6) -- (b12);
		\draw(b6) -- (b13);
		\draw(b7) -- (b12);
		\draw(b7) -- (b14);
		\draw(b8) -- (b13);
		\draw(b8) -- (b14);
		\draw(b9) -- (b12);
		\draw(b9) -- (b15);
		\draw(b10) -- (b13);
		\draw(b10) -- (b15);
		\draw(b11) -- (b14);
		\draw(b11) -- (b15);
		\draw(b12) -- (b16);
		\draw(b13) -- (b16);
		\draw(b14) -- (b16);
		\draw(b15) -- (b16);
	\end{tikzpicture}
	\caption{The Hasse diagram of $\mathcal{B}_{4}^{(1)}$.}
	\label{fig:41_balloon}
\end{figure}

Before we enumerate the Galois connections between $m$-balloons and $(n+1)$-chains, we recall the 
definitions. A \alert{Galois connection} between two posets 
$(P,\leq_{P})$ and $(Q,\leq_{Q})$ is a pair $(\varphi,\psi)$ of maps
\begin{displaymath}
	\varphi: P\to Q\quad\text{and}\quad\psi:Q\to P,
\end{displaymath}
satisfying 
\begin{align*}
	& p_{1} \leq_{P} p_{2} \quad\text{implies}\quad\varphi p_{1}\geq_{Q}\varphi p_{2},\\
	& q_{1} \leq_{Q} q_{2} \quad\text{implies}\quad\psi q_{1}\geq_{P}\psi q_{2},\\
	& p\leq_{P}\psi\varphi p, \quad\text{and}\quad
	 q\leq_{Q}\varphi\psi q,
\end{align*}
for all $p,p_{1},p_{2}\in P$ and $q,q_{1},q_{2}\in Q$. 
Recall that, given formal contexts $\KK_{1}=(G,M,I)$ and $\KK_{2}=(H,N,J)$, a relation 
$R\subseteq G\times H$, is called \alert{dual bond from $\KK_{1}$ to $\KK_{2}$} if for every 
$g\in G$, the set $g^{R}$ is an extent of $\KK_{2}$ and for every $h\in H$, the set $h^{R}$ is an 
extent of $\KK_{1}$. In other words, $R$ is a dual bond from $\KK_{1}$ to $\KK_{2}$ if and only if $R$ is a 
bond from $\KK_{1}$ to the dual\footnote{Let $\KK=(G,M,I)$ be a formal context. The dual context $\KK^{d}$ 
of $\KK$ is given by $(M,G,I^{-1})$ and satisfies $\CL(\KK^{d})\cong\CL(\KK)^{d}$, where $\CL(\KK)^{d}$ is 
the (order-theoretic) dual of the lattice $\CL(\KK)$.} context $\KK_{2}^{d}$.
In the case, where the posets $(P,\leq_{P})\cong\CL(\KK_{1})$ and $(Q,\leq_{Q})\cong\CL(\KK_{2})$ are 
concept lattices, we can interpret the Galois connections between $(P,\leq_{P})$ and $(Q,\leq_{Q})$ as dual 
bonds from $\KK_{1}$ to $\KK_{2}$ as described in the following theorem.

\begin{theorem}[\cite{ganter99formal}*{Theorem~53}]
  \label{thm:galois_connections}
	Let $(G,M,I)$ and $(H,N,J)$ be formal contexts. For every dual bond $R\subseteq G\times H$, the 
	maps
	\begin{align*}
		\varphi_{R}\bigl(X,X^{I}\bigr)=\bigl(X^{R},X^{RJ}\bigr),\quad\text{and}\quad
		  \psi_{R}\bigl(Y,Y^{J}\bigr)=\bigl(Y^{R},Y^{RI}\bigr),
	\end{align*}
	where $X$ and $Y$ are extents of $(G,M,I)$ respectively $(H,N,J)$, form a Galois connection 
	between $\CL(G,M,I)$ and $\CL(H,N,J)$. Moreover, every Galois connection $(\varphi,\psi)$ 
	induces a dual bond from $(G,M,I)$ to $(H,N,J)$ by 
	\begin{align*}
		R_{(\varphi,\psi)}=\bigl\{(g,h)\mid\gamma g\leq\psi\gamma h\bigr\}=
		  \bigl\{(g,h)\mid\gamma h\leq\varphi\gamma g\bigr\},
	\end{align*}
	where $\gamma$ is the map defined in \eqref{eq:maps}. In particular, we have
	\begin{align*}
		\varphi_{R_{(\varphi,\psi)}}=\varphi,\quad\psi_{R_{(\varphi,\psi)}}=\psi,\quad
		  \text{and}\quad R_{(\varphi_{R},\psi_{R})}=R.
	\end{align*}
\end{theorem}

Since chains are self-dual, the previous theorem implies that every Galois connection between an 
$(n+1)$-chain and an $m$-balloon corresponds to a bond from $(C,C,\not\geq_{\cc})$ to $(S,S,\not\geq_{\fs})$. 
In view of Proposition~\ref{prop:classification_mergings} this means that every Galois connection between an 
$(n+1)$-chain and an $m$-balloon corresponds to a proper merging of $\fs$ and $\cc$ which is 
of the form $(\emptyset,T)$. These are relatively easy to enumerate as our next proposition shows.

\begin{proposition}
  \label{prop:half_empty_bonds}
	Let $\fs$ be an $m$-star and let $\cc$ be an $n$-chain. The number of proper mergings
	of $\fs$ and $\cc$ which are of the form $(\emptyset,T)$ is $\sum_{k=1}^{n+1}{k^{m}}$.
\end{proposition}
\begin{proof}
	Let $(\emptyset,T)$ be a proper merging of $\fs$ and $\cc$. Thus, $T\subseteq C\times S$ is a bond 
	from $(C,C,\not\geq_{\cc})$ to $(S,S,\not\geq_{\fs})$. This means, for every $c\in C$, the row 
	$c^{T}$ is an intent of $(S,S,\not\geq_{\fs})$, and thus must be either the set $S$ or a set of the 
	form $S\setminus(B\cup\{s_{0}\})$ for some $B\subseteq S\setminus\{s_{0}\}$. Moreover, for every 
	$s\in S$, the column $s^{T}$ is an extent of $(C,C,\not\geq_{\cc})$, and thus must be of the form 
	$\{c_{1},c_{2},\ldots,c_{i-1}\}$ for some $i\in\{1,2,\ldots,n+1\}$. (The case $i=1$ is to be 
	interpreted as the empty set.) 
	
	Since $T$ is a bond from $(C,C,\not\geq_{\cc})$ to $(S,S,\not\geq_{\fs})$, we notice that if 
	$c_{i}\;T\;s_{j}$, then $c_{k}\;T\;s_{j}$ for every $k\in\{1,2,\ldots,i\}$. In particular, if the 
	$i$-th row of $T$ is a full row, then every row above the $i$-th row is also a full row. Furthermore,
	if $c_{i}\;T\;s_{0}$, then $c_{i}\;T\;s_{k}$ for every $k\in\{0,1,\ldots,m\}$, since the only intent
	of $(S,S,\not\geq_{\fs})$ that contains $\{s_{0}\}$ is $S$ itself. 
	
	\begin{figure}\scriptsize
		\centering
		\begin{tabular}{c@{}l}
			\begin{tabular}{|c||c|c|c|c|c|}
				\hline
				$T$ & $s_{0}$ & $s_{1}$ & $s_{2}$ & $\cdots$ & $s_{m}$ \\
				\hline\hline
				$c_{1}$ & $\times$ & $\times$ & $\times$ & $\cdots$ & $\times$ \\
				\hline
				$c_{2}$ & $\times$ & $\times$ & $\times$ & $\cdots$ & $\times$ \\
				\hline
				$\vdots$ & $\vdots$ & $\vdots$ & $\vdots$ & $\vdots$ & $\vdots$ \\
				\hline
				$c_{k}$ & $\times$ & $\times$ & $\times$ & $\cdots$ & $\times$ \\
				\hline
				$c_{k+1}$ & & & & & \\
				\hline
				$c_{k+2}$ & & & & & \\
				\hline
				$\vdots$ & & & & & \\
				\hline
				$c_{n}$ & & & & & \\
				\hline
			\end{tabular}
			& 
			$\begin{array}{l}
				\\
				\rBrace{6.8ex}{\parbox{7em}{fixed block,\\ bond property satisfied}} \\ 
				\rBrace{6.8ex}{\parbox{7em}{no crosses in\\ first column}} \\
			\end{array}$
		\end{tabular}
		\caption{Illustration of the situation with $k$ full rows in $T$.}
		\label{fig:restriction}
	\end{figure}

	Now let $k\in\{1,2,\ldots,n\}$ be the maximal index such that $c_{k}^{T}=S$, and write
	$C_{n-k}=\{c_{k+1},c_{k+2},\ldots,c_{n}\}$. We have just seen 
	that this implies that $c_{j}^{T}=S$ for $j\leq k$, and $(c_{j},s_{0})\notin T$ for $j>k$. Hence,
	$T$ is a bond from $\bigl(C,C,\not\geq_{\cc})$ to $(S,S,\not\geq_{\fs}\bigr)$ if and only if the 
	restriction of $T$ to $C_{n-k}\times \bigl(S\setminus\{s_{0}\}\bigr)$ is a bond from 
	$\bigl(C_{n-k},C_{n-k},\not\geq_{\cc}\bigr)$ to 
	$\bigl(S\setminus\{s_{0}\},S\setminus\{s_{0}\},\not\geq_{\fs}\bigr)$. See 
	Figure~\ref{fig:restriction} for an illustration. Clearly, $\CL(C_{n-k},C_{n-k},\not\geq_{\cc})$ is
	isomorphic to an $(n-k+1)$-chain and $\CL(S\setminus\{s_{0}\},S\setminus\{s_{0}\},\not\geq_{\fs})$ is
	isomorphic to the Boolean lattice $\mathcal{B}_{m}$. It follows from 
	\cite{muehle12counting}*{Proposition~5.8} that the number of bonds from 
	$(C_{n-k},C_{n-k},\not\geq_{\cc})$ to $(S\setminus\{s_{0}\},S\setminus\{s_{0}\},\not\geq_{\fs})$ is
	$(n-k+1)^{m}$. 
	
	The number $g(m,n)$ of proper mergings of $\fs$ and $\cc$ which are of the form $(\emptyset,T)$ is 
	now the sum over all proper mergings of $\fs$ and $\cc$ which are of the form $(\emptyset,T)$, and 
	where the first $k$ rows of $T$ are full rows. We obtain
	\begin{displaymath}
		g(m,n)=\sum_{k=0}^{n}{(n-k+1)^{m}}=\sum_{k=1}^{n+1}{k^{m}},
	\end{displaymath}
	as desired.
\end{proof}

\begin{proof}[Proof of Proposition~\ref{prop:galois_ballon_chain}]
	This follows immediately from Proposition~\ref{prop:half_empty_bonds}.
\end{proof}

Appendix~\ref{app:illustration} lists the proper mergings of an $3$-star and a $1$-chain that are of the
form $(\emptyset,T)$, and the corresponding Galois connections between an $3$-balloon and a $2$-chain. 

\section*{Acknowledgements}
  \label{sec:acknowledgements}
The author is very grateful to three anonymous referees for their careful reading and their helpful remarks on 
presentation and content of the article.

\appendix

\section{Proof of Lemma~\ref{lem:equality}}
  \label{app:proof}
Recall from \eqref{eq:init} that putting \eqref{eq:cardinality_decomposition_1}, 
\eqref{eq:cardinality_decomposition_2} and Lemmas~\ref{lem:count_v1}--\ref{lem:count_valid} together, yields
\begin{displaymath}
	F_{\fs\!\cc}(m,n)
	  = \sum_{k_{1}=1}^{n+1}{F_{V_{1}}(m,n,k_{1})\sum_{k_{2}=0}^{k_{1}-1}\sum_{l=0}^{k_{2}}
	  F_{V_{2}}(m,k_{2},l)\cdot\left(k_{1}(l+1)-\binom{l+1}{2}\right)},
\end{displaymath}
where
\begin{align*}
	F_{V_{1}}(m,n,k_{1}) & = (n+2-k_{1})^{m}-(n+1-k_{1})^{m},\quad\mbox{and}\\
	F_{V_{2}}(m,k_{2},l) & = 
	  \begin{cases}
		  (k_{2}-l+1)^{m}-2(k_{2}-l)^{m}+(k_{2}-l-1)^{m}, & \mbox{if}\;l<k_{2}\\
		  1, & \mbox{if}\;l=k_{2}.
	  \end{cases}
\end{align*}
Recall further that 
\begin{displaymath}
	C(m,n) = \sum_{k=1}^{n}{k^{m}(n-k+2)^{m+1}},
\end{displaymath}
and we want to show that $F_{\fs\!\cc}(m,n)=C(m,n+1)$. Let us first focus on the term
\begin{align*}
	A(m,k_{1},k_{2}) & = \sum_{l=0}^{k_{2}-1}{F_{V_{2}}(m,k_{2},l)\cdot
	  \left(k_{1}(l+1)-\binom{l+1}{2}\right)}\\
	& = \sum_{l=0}^{k_{2}-1}{\Bigl((k_{2}-l+1)^m-2(k_{2}-l)^m+(k_{2}-l-1)^m\Bigr)
	  \cdot\left(k_{1}(l+1)-\binom{l+1}{2}\right)}.
\end{align*}
We can convince ourselves quickly that the following identities are true:
\begin{align*}
	k_{1}(l+1)-\binom{l+1}{2} & = k_{1}(l+2)-\binom{l+2}{2}+l+1-k_{1},\quad\mbox{and}\\
	k_{1}(l+1)-\binom{l+1}{2} & = k_{1}(l+3)-\binom{l+3}{2}+2l+3-2k_{1}.
\end{align*}
Thus, we can write
\begin{align*}
	A(m,k_{1},k_{2}) & = \sum_{l=0}^{k_{2}-1}{(k_{2}-l+1)^{m}
	  \cdot\left(k_{1}(l+1)-\binom{l+1}{2}\right)}\\
	& \kern1cm - 2\sum_{l=0}^{k_{2}-1}{(k_{2}-l)^{m}\cdot\left(k_{1}(l+1)-\binom{l+1}{2}\right)}\\
	& \kern1cm + \sum_{l=0}^{k_{2}-1}{(k_{2}-l-1)^{m}\cdot\left(k_{1}(l+1)-\binom{l+1}{2}\right)}\\
	& = \sum_{l=0}^{k_{2}-1}{(k_{2}-l+1)^{m}\cdot\left(k_{1}(l+1)-\binom{l+1}{2}\right)}\\
	& \kern1cm - 2\sum_{l=0}^{k_{2}-1}{(k_{2}-(l+1)+1)^{m}\cdot
		\left(k_{1}(l+2)-\binom{l+2}{2}+l+1-k_{1}\right)}\\
	& \kern1cm + \sum_{l=0}^{k_{2}-1}{(k_{2}-(l+2)+1)^{m}\cdot
		\left(k_{1}(l+3)-\binom{l+3}{2}+2l+3-2k_{1}\right)}.
\end{align*}
If we define $\varphi(m,k_{1},k_{2},l)=(k_{2}-l+1)^{m}\cdot\left(k_{1}(l+1)-\binom{l+1}{2}\right)$, then we 
obtain
\begin{align*}
	A(m,k_{1},k_{2}) & = \sum_{l=0}^{k_{2}-1}{\varphi(m,k_{1},k_{2},l)}\\
	& \kern1cm - 2\sum_{l=0}^{k_{2}-1}{\varphi(m,k_{1},k_{2},l+1)}
		-2\sum_{l=0}^{k_{2}-1}{(k_{2}-l)^{m}\cdot(l+1-k_{1})}\\
	& \kern1cm +\sum_{l=0}^{k_{2}-1}{\varphi(m,k_{1},k_{2},l+2)}
		+\sum_{l=0}^{k_{2}-1}{(k_{2}-l-1)^{m}\cdot\bigl(2(l+1-k_{1})+1\bigr)}\\
	& = \sum_{l=0}^{k_{2}-1}{\varphi(m,k_{1},k_{2},l)} - 2\sum_{l=0}^{k_{2}-1}{\varphi(m,k_{1},k_{2},l+1)}
		+\sum_{l=0}^{k_{2}-1}{\varphi(m,k_{1},k_{2},l+2)}\\
	& \kern1cm -2\sum_{l=0}^{k_{2}-1}{(k_{2}-l)^{m}\cdot(l+1-k_{1})}
		+\sum_{l=0}^{k_{2}-1}{(k_{2}-l-1)^{m}\cdot\bigl(2(l+1-k_{1})+1\bigr)}\\
	& = \sum_{l=0}^{k_{2}-1}{\varphi(m,k_{1},k_{2},l)} - 2\sum_{l=0}^{k_{2}-1}{\varphi(m,k_{1},k_{2},l+1)}
		+\sum_{l=0}^{k_{2}-1}{\varphi(m,k_{1},k_{2},l+2)}\\
	& \kern1cm +\sum_{l=0}^{k_{2}-1}{(k_{2}-l)^{m}\cdot(2k_{1}-2l-2)}
		+\sum_{l=0}^{k_{2}-1}{(k_{2}-l-1)^{m}\cdot\bigl(2l+3-2k_{1}\bigr)}.\\
\end{align*}
Let us now simplify the terms not involving $\varphi$.
\begin{align*}
	\psi(m,k_{1},k_{2}) & = \sum_{l=0}^{k_{2}-1}{(k_{2}-l)^{m}\cdot(2k_{1}-2l-2)} 
		+\sum_{l=0}^{k_{2}-1}{(k_{2}-l-1)^{m}\cdot\bigl(2l+3-2k_{1}\bigr)}\\
	& = \Bigl(k_{2}^{m}(2k_{1}-2)+(k_{2}-1)^{m}(2k_{1}-4)+\cdots+1^{m}(2k_{1}-2k_{2})\Bigr)\\
	& \kern1cm 
		+ \Bigl((k_{2}-1)^{m}(3-2k_{1})+(k_{2}-2)^{m}(5-2k_{1})+\cdots+1^{m}(2k_{2}-1-2k_{1})\Bigr)\\
	& = k_{2}^{m}(2k_{1}-2)-(k_{2}-1)^{m}-(k_{2}-2)^{m}-\cdots-1^m\\
	& = k_{2}^{m}(2k_{1}-2)-\sum_{l=1}^{k_{2}-1}{l^{m}}.
\end{align*}
Applying this identity and shifting indices yields
\begin{align*}
	A(m,k_{1},k_{2}) & = \sum_{l=0}^{k_{2}-1}{\varphi(m,k_{1},k_{2},l)} 
	  - 2\sum_{l=1}^{k_{2}}{\varphi(m,k_{1},k_{2},l)}
	  +\sum_{l=2}^{k_{2}+1}{\varphi(m,k_{1},k_{2},l)}\\
	& \kern1cm + \psi(m,k_{1},k_{2})\\
	& = \varphi(m,k_{1},k_{2},0)-\varphi(m,k_{1},k_{2},1)-\varphi(m,k_{1},k_{2},k_{2})
	  +\varphi(m,k_{1},k_{2},k_{2}+1)\\
	& \kern1cm + \psi(m,k_{1},k_{2})\\
	& = (k_{2}+1)^{m}k_{1} - k_{2}^{m}(2k_{1}-1) - k_{1}(k_{2}+1) + \binom{k_{2}+1}{2}\\
	& \kern1cm + k_{2}^{m}(2k_{1}-2)-\sum_{l=1}^{k_{2}-1}{l^{m}}\\
	& = k_{1}(k_{2}+1)^{m}-k_{1}(k_{2}+1) + \binom{k_{2}+1}{2} - \sum_{l=1}^{k_{2}}{l^{m}}.
\end{align*}
So far, we have shown that
\begin{align}\label{eq:step1}
	F_{\fs\!\cc}(m,n) & = \sum_{k_{1}=1}^{n+1}{\Bigl((n+2-k_{1})^{m}-(n+1-k_{1})^{m}\Bigr)}\\
\nonumber	
	& \kern1cm\cdot\sum_{k_{2}=0}^{k_{1}-1}\left(k_{1}(k_{2}+1)^{m}-k_{1}(k_{2}+1) + 
	  \binom{k_{2}+1}{2} - \sum_{l=1}^{k_{2}}{l^{m}}\right.\\
\nonumber	
	& \kern2cm + k_{1}(k_{2}+1)-\binom{k_{2}+1}{2}\left.\vphantom{\sum_{k_{2}=0}^{k_{1}-1}}\right)\\
\nonumber	
	& = \sum_{k_{1}=1}^{n+1}{\Bigl((n+2-k_{1})^{m}-(n+1-k_{1})^{m}\Bigr)}\\
\nonumber	
	& \kern1cm\cdot\sum_{k_{2}=0}^{k_{1}-1}\left(k_{1}(k_{2}+1)^{m} - \sum_{l=1}^{k_{2}}{l^{m}} + \right)\\
\nonumber	
	& = \sum_{k_{1}=1}^{n+1}{\Bigl((n+2-k_{1})^{m}-(n+1-k_{1})^{m}\Bigr)}\\
\nonumber
	& \kern1cm\cdot\left(\sum_{k_{2}=0}^{k_{1}-1}{k_{1}(k_{2}+1)^{m}}
	  - \sum_{k_{2}=0}^{k_{1}-1}{\sum_{l=1}^{k_{2}}{l^{m}}}\right).
\end{align}
We may now simplify the inner double sum:
\begin{align*}
	\sum_{k_{2}=0}^{k_{1}-1}\sum_{l=1}^{k_{2}}{l^{m}} 
	  & = 0+\sum_{l=1}^{1}{l^{m}}+\sum_{l=1}^{2}{l^{m}}+\cdots+\sum_{l=1}^{k_{1}-1}{l^{m}}\\
	& = k_{1}0^{m} + (k_{1}-1)1^{m} + (k_{1}-2)2^{m} + \cdots + 1(k_{1}-1)^{m}\\
	& = \sum_{k_{2}=0}^{k_{1}-1}{(k_{1}-k_{2})k_{2}^{m}}.
\end{align*}
If this is substituted in \eqref{eq:step1}, we obtain
\begin{align}\label{eq:step2}
	F_{\fs\!\cc}(m,n) & = \sum_{k_{1}=1}^{n+1}{\Bigl((n+2-k_{1})^{m}-(n+1-k_{1})^{m}\Bigr)}\cdot
	  \sum_{k_{2}=0}^{k_{1}-1}{k_{1}(k_{2}+1)^{m}}\\
\nonumber
	& \kern1cm - \sum_{k_{1}=1}^{n+1}{\Bigl((n+2-k_{1})^{m}-(n+1-k_{1})^{m}\Bigr)}\cdot
	  \sum_{k_{2}=0}^{k_{1}-1}{(k_{1}-k_{2})k_{2}^{m}} \\
\nonumber
	& = \sum_{k_{1}=1}^{n+1}{\Bigl((n+2-k_{1})^{m}-(n+1-k_{1})^{m}\Bigr)}\cdot
	  \sum_{k_{2}=0}^{k_{1}-1}{k_{1}(k_{2}+1)^{m}}\\
\nonumber
	& \kern1cm - \sum_{k_{1}=1}^{n+1}{\Bigl((n+2-k_{1})^{m}-(n+1-k_{1})^{m}\Bigr)}\cdot
	  \sum_{k_{2}=0}^{k_{1}-1}{k_{1}k_{2}^{m}}\\
\nonumber
	& \kern1cm + \sum_{k_{1}=1}^{n+1}{\Bigl((n+2-k_{1})^{m}-(n+1-k_{1})^{m}\Bigr)}\cdot
	  \sum_{k_{2}=0}^{k_{1}-1}{k_{2}^{m+1}}\\
\nonumber
	& = \sum_{k_{1}=1}^{n+1}{\Bigl((n+2-k_{1})^{m}-(n+1-k_{1})^{m}\Bigr)}\cdot
	  \sum_{k_{2}=1}^{k_{1}}{k_{1}k_{2}^{m}}\\
\nonumber
	& \kern1cm - \sum_{k_{1}=1}^{n+1}{\Bigl((n+2-k_{1})^{m}-(n+1-k_{1})^{m}\Bigr)}\cdot
	  \sum_{k_{2}=0}^{k_{1}-1}{k_{1}k_{2}^{m}}\\
\nonumber
	& \kern1cm + \sum_{k_{1}=1}^{n+1}{\Bigl((n+2-k_{1})^{m}-(n+1-k_{1})^{m}\Bigr)}\cdot
	  \sum_{k_{2}=0}^{k_{1}-1}{k_{2}^{m+1}}\\
\nonumber
	& = \sum_{k_{1}=1}^{n+1}{\Bigl((n+2-k_{1})^{m}-(n+1-k_{1})^{m}\Bigr)}\cdot k_{1}^{m+1}\\
\nonumber
	& \kern1cm + \sum_{k_{1}=1}^{n+1}{\Bigl((n+2-k_{1})^{m}-(n+1-k_{1})^{m}\Bigr)}\cdot
	  \sum_{k_{2}=0}^{k_{1}-1}{k_{2}^{m+1}}.
\end{align}
It is easy to check the identities
\begin{multline*}
	\sum_{k_{1}=1}^{n+1}{\Bigl((n+2-k_{1})^{m}-(n+1-k_{1})^{m}\Bigr)}\cdot k_{1}^{m+1} \\
	  = \sum_{k_{1}=1}^{n+1}{k_{1}^{m}\Bigl((n+2-k_{1})^{m+1}-(n+1-k_{1})^{m+1}\Bigr)},
\end{multline*}
and 
\begin{align*}
	\sum_{k_{1}=1}^{n+1}{\Bigl((n+2-k_{1})^{m}-(n+1-k_{1})^{m}\Bigr)}\cdot
	  \sum_{k_{2}=0}^{k_{1}-1}{k_{2}^{m+1}} = \sum_{k_{1}=1}^{n+1}{k_{1}^{m}(n+1-k_{1})^{m+1}}.
\end{align*}
Thus, substituting these in \eqref{eq:step2}, we obtain
\begin{align*}
	F_{\fs\!\cc}(m,n) & = \sum_{k_{1}=1}^{n+1}{\Bigl((n+2-k_{1})^{m}-(n+1-k_{1})^{m}\Bigr)}\cdot k_{1}^{m+1}\\
	& \kern1cm + \sum_{k_{1}=1}^{n+1}{\Bigl((n+2-k_{1})^{m}-(n+1-k_{1})^{m}\Bigr)}\cdot
	  \sum_{k_{2}=0}^{k_{1}-1}{k_{2}^{m+1}}\\
	& = \sum_{k_{1}=1}^{n+1}{k_{1}^{m}\Bigl((n+2-k_{1})^{m+1}-(n+1-k_{1})^{m+1}\Bigr)} 
	  + \sum_{k_{1}=1}^{n+1}{k_{1}^{m}(n+1-k_{1})^{m+1}}\\
	& = \sum_{k_{1}=1}^{n+1}{k_{1}^{m}\Bigl((n+2-k_{1})^{m+1}
	  - (n+1-k_{1})^{m+1} 
	  + (n+1-k_{1})^{m+1}\Bigr)}\\
	& = \sum_{k_{1}=1}^{n+1}{k_{1}^{m}(n+2-k_{1})^{m+1}}\\
	& = C(m,n+1),
\end{align*}
as desired. \qed

\section{Illustration of Proposition~\ref{prop:galois_ballon_chain}}
  \label{app:illustration}
\begin{remark}
	Let $(\emptyset,T)$ be a proper merging of an $m$-star $(S,\leq_{\fs})$ and an $n$-chain 
	$(C,\leq_{\cc})$. In order to produce the corresponding Galois connection, we define a dual bond
	$\hat{T}$ between $(S,S,\not\geq_{\fs})$ and $(C,C,\not\geq_{\cc})$ as follows: for every 
	$i\in\{1,2,\ldots,n\}$, we define
	\begin{displaymath}
		c_{i}^{\hat{T}}=
			\begin{cases}
				S\setminus\{c_{n+1-i}^{T}\} & \mbox{if}\;c_{n+1-i}^{T}\neq S,\quad\mbox{and}\\
				\emptyset & \mbox{otherwise}.
			\end{cases}
	\end{displaymath}
\end{remark}

\begin{longtable}{c|c|c|c}
	$(R,T)$ & $\hat{T}$ & $\varphi_{\hat{T}}$ & $\psi_{\hat{T}}$\\
	\hline
	\sNine{1}{.8} & \raisebox{.5cm}{\thEight} & \PhiEight & \PsiEight \\
	\hline
	\sThirteen{1}{.8} & \raisebox{.5cm}{\thSeven} & \PhiSeven & \PsiSeven \\
	\hline
	\sFourteen{1}{.8} & \raisebox{.5cm}{\thSix} & \PhiSix & \PsiSix\\
	\hline
	\sFifteen{1}{.8} & \raisebox{.5cm}{\thFive} & \PhiFive & \PsiFive \\
	\hline
	\sNineteen{1}{.8} & \raisebox{.5cm}{\thFour} & \PhiFour & \PsiFour \\
	\hline
	\sTwenty{1}{.8} & \raisebox{.5cm}{\thThree} & \PhiThree & \PsiThree\\
	\hline
	\sTwentyone{1}{.8} & \raisebox{.5cm}{\thTwo} & \PhiTwo & \PsiTwo \\
	\hline
	\sTwentythree{1}{.8} & \raisebox{.5cm}{\thOne} & \PhiOne & \PsiOne \\
	\hline
	\sTwentyfour{1}{.8} & \raisebox{.5cm}{\thNine} & \PhiNine & \PsiNine \\
\end{longtable}

\section{Illustration of Farley's Bijection}
  \label{app:illustration_bijection}

\begin{longtable}{c|c|c|c}
	$\zeta\in O\!P\bigl(\cc_{n},\PP_{1,3,2})$ & $R_{\zeta}$ & $T_{\zeta}$ 
	  & $\bigl(\{s_{0},s_{1},s_{2},c_{1}\},\leq_{R_{\zeta},T_{\zeta}}\bigr)$\\
	\hline
	\ZetaOne & \raisebox{1.5cm}{\rzOne} & \raisebox{1.5cm}{\tzOne} & \raisebox{1cm}{\pzOne{1}{.8}}\\
	\hline
	\ZetaTwo & \raisebox{1.5cm}{\rzThree} & \raisebox{1.5cm}{\tzOne} & \raisebox{1cm}{\pzTwo{1}{.8}}\\
	\hline
	\ZetaThree & \raisebox{1.5cm}{\rzThree} & \raisebox{1.5cm}{\tzTwo} & \raisebox{1cm}{\pzThree{1}{.8}}\\
	\hline
	\ZetaFour & \raisebox{1.5cm}{\rzFour} & \raisebox{1.5cm}{\tzOne} & \raisebox{1cm}{\pzFour{1}{.8}}\\
	\hline
	\ZetaFive & \raisebox{1.5cm}{\rzFive} & \raisebox{1.5cm}{\tzOne} & \raisebox{1cm}{\pzFive{1}{.8}}\\
	\hline
	\ZetaSix & \raisebox{1.5cm}{\rzFour} & \raisebox{1.5cm}{\tzTwo} & \raisebox{1cm}{\pzSix{1}{.8}}\\
	\hline
	\ZetaSeven & \raisebox{1.5cm}{\rzFive} & \raisebox{1.5cm}{\tzTwo} & \raisebox{1cm}{\pzSeven{1}{.8}}\\
	\hline
	\ZetaEight & \raisebox{1.5cm}{\rzTwo} & \raisebox{1.5cm}{\tzOne} & \raisebox{1cm}{\pzEight{1}{.8}}\\
	\hline
	\ZetaNine & \raisebox{1.5cm}{\rzTwo} & \raisebox{1.5cm}{\tzTwo} & \raisebox{1cm}{\pzNine{1}{.8}}\\
	\hline
	\ZetaTen & \raisebox{1.5cm}{\rzTwo} & \raisebox{1.5cm}{\tzThree} & \raisebox{1cm}{\pzTen{1}{.8}}\\
	\hline
	\ZetaEleven & \raisebox{1.5cm}{\rzTwo} & \raisebox{1.5cm}{\tzFour} & \raisebox{1cm}{\pzEleven{1}{.8}}\\
	\hline
	\ZetaTwelve & \raisebox{1.5cm}{\rzTwo} & \raisebox{1.5cm}{\tzFive} & \raisebox{1cm}{\pzTwelve{1}{.8}}\\
\end{longtable}

\bibliography{literature}
  \label{sec:references}

\end{document}